\documentclass{article}
\usepackage[margin=1.5in]{geometry}
\usepackage{amsmath}
\usepackage{amsthm}
\usepackage{amssymb}
\usepackage{enumerate}
\usepackage{graphicx}
\graphicspath{{img/}}
\usepackage{subcaption}
\usepackage{float}

\newtheorem{theorem}{Theorem}[section]

\newtheorem{lemma}[theorem]{Lemma}
\newtheorem{definition}{Definition}[section]
\newtheorem{remark}{Remark}
\newtheorem{assumption}{Assumption}[section]

\renewcommand{\Re}{\operatorname{Re}}
\newcommand{\tstep}{{\kappa}}
\newcommand{\tder}[1]{\partial_\tstep^{{#1}}}
\newcommand{\ctder}[1]{ \partial_{\tstep\ast}^{{#1}}}
\newcommand{\sder}{{\bar \partial}_t}
\renewcommand{\u}{u^h}
\newcommand{\vh}{v^h}
\newcommand{\w}{w^h}
\newcommand{\e}{e^h}

\newcommand{\Cinv}{C_{\text{inv}}}
\newcommand{\ag}{a_\gamma}

\begin{document}
\title{Numerical analysis of a wave equation for lossy media obeying a frequency power law}
\author{Katherine Baker \thanks{Maxwell Institute for Mathematical Sciences, Department of Mathematics, Heriot-Watt University, Edinburgh, UK, EH14 4AS.
({\tt kb54@hw.ac.uk})} \and Lehel Banjai \thanks{Maxwell Institute for Mathematical Sciences, Department of Mathematics, Heriot-Watt University, Edinburgh, UK, EH14 4AS.
({\tt l.banjai@hw.ac.uk})}}

\maketitle

\begin{abstract}
{We study a wave equation with a nonlocal time fractional damping term that models the effects of acoustic attenuation characterized by a frequency dependence power law. First we prove existence of a unique solution to this equation with particular attention paid to the handling of the fractional derivative. Then we derive an explicit time stepping scheme based on the finite element method in space and a combination of convolution quadrature and second order central differences in time. We conduct a full error analysis of the mixed time discretization and in turn the fully space time discretized scheme. Error estimates are given for both smooth solutions and solutions with a singularity at $t = 0$ of a type that is typical for equations involving fractional time-derivatives. A number of numerical results are presented to support the error analysis. \\

} 
\end{abstract}


\section{Introduction}

We are interested in  the initial boundary problem given by a wave equation with the addition of a time fractional damping term on a bounded Lipschitz domain $\Omega \subset \mathbb{R}^d$ with boundary $\partial\Omega$,
\begin{equation}\label{equ:PDE}
  \begin{aligned}
\frac1{c^2}\partial_t^2 u - \Delta u + \frac{\ag}{c} \partial_t^{\gamma+1} u &= f, & &\text{in } \Omega \times [0,T] \\
u(\cdot,0) &= u_0 & &\text{in } \Omega \\
\partial_t u(\cdot, 0) &= v_0 & &\text{in } \Omega \\
u &= 0 \quad & &\text{on }\partial \Omega \times[0,T],
\end{aligned}
\end{equation}
Here  $\gamma \in (-1,1)\setminus\{0\}$, $c(x) > 0$ is the wave speed, and 
\begin{equation}
  \label{eq:agamma}
 \ag = -\alpha_0\frac{4}{\pi} \Gamma(-\gamma-1)\Gamma(\gamma+2) \cos((\gamma+1)\pi/2),
\end{equation}
for some constant $\alpha_0 > 0$.
Note that $\ag > 0$ for $\gamma \in (-1,1)$ and that $\ag\rightarrow \infty$ as $\gamma \rightarrow \pm 1$. In this problem $\partial_t^\gamma u$ denotes the Caputo fractional derivative; see Definition \ref{def:caputoderiv} and we only consider $\gamma \neq 0$ throughout. For the rest of the paper, to simplify notation we set $c \equiv 1$, however we do keep track of the constant $a_\gamma$.

The interest in this problem stems from the  modelling of acoustic attenuation that occurs as a wave propagates through lossy media \cite[Chapter 4]{szabo:book}.
In applications, this is applied to modelling high intensity focused ultrasound therapy (HIFU) where the lossy media is biological tissue. HIFU is a noninvasive and nonionising medical treatment that focuses multiple high intensity acoustic pressure waves on a region, ablating it away \cite{terhaar2016}. It is known that acoustic attenuation obeys a frequency dependence characterized by the following power law,
\begin{equation}
\label{equ:freqpowerlaw}
S(\overrightarrow{x} + \Delta \overrightarrow{x}) = S(\overrightarrow{x}) e^{-\alpha(\omega)| \overrightarrow{x}|}
\end{equation}
where $S$ is the amplitude, $\Delta \overrightarrow{x}$ is the wave propagation distance, $\omega$ denotes frequency and $\alpha(\omega)$ the attenuation coefficient which defined by
\[
\alpha(\omega) = \alpha_0 |\omega |^\gamma
\]
for  $\gamma$ the frequency power exponent and $\alpha_0$ a constant relating to the media \cite{chen_holm2003}. Values of $\gamma$ have been determined by many experiments and field measurements, their results conclude that for most media $\gamma \in (0,2)$ \cite{duck:book}.  In this paper we restrict our attention to $\gamma \in (-1,1)$. For the range $\gamma \in (1,2)$ (and potentially also $(0,1)$) we prefer changing the weak damping term $\partial_t^{\gamma+1}$ with the strong damping of the form $(-\Delta) \partial_t^{\gamma-1}$  \cite{Caputo1967,holm2013,Kelly2016,wismer2006}. The techniques developed in this paper, can be used to analyze the strongly damped case for the full range $\gamma \in (0,2)$; see \cite{katiephd}. Another model that we can consider is the nonlinear fractional Westervelt equation \cite{holm2013}, this contains the strong damping term mentioned above and describes the nonlinear behaviour that the high intensity focusing causes during HIFU. Similar models appear also in fractional order viscoelasticity, similar to both the strong damping case (for $\gamma \in (0,1)$) \cite{larssonfrac1,Saedpanah} and the weak damping case  \cite{EndreZener}.

\cite{szabo1994} derived a wave equation with a convolution type operator to incorporate the effects of acoustic attenuation and showed that it adheres to the frequency dependence power law \eqref{equ:freqpowerlaw}. \cite{chen_holm2003} praise this model for its simplicity, due to it only containing two parameters, but criticise it for being difficult to implement initial conditions. To overcome this they refined the model to include the Caputo fractional derivative as the damping term. Furthermore within their paper they show that even with this modification the solutions still obey the power law.

Issues that arise within this model largely stem from the convolution based fractional derivative that requires us to store the full history when numerically solving with a time stepping scheme, making it expensive to compute, particularly in 3D. In the literature this is countered by replacing the fractional time derivative with a fractional Laplacian \cite{chen_holm2004,treeby2010}. However, using efficient quadrature methods \cite{banjai2019} will also be sufficient in reducing memory requirements and allows us to remain using the simpler operator. 

The paper is structured in the following way. Section~\ref{sec:dampedwaveequationmodel} begins by outlining necessary definitions and lemmas required to prove existence and uniqueness of solutions to \eqref{equ:PDE}. In Section~\ref{sec:fullydiscretesystem} we develop a numerical scheme to approximate solutions of \eqref{equ:PDE} using finite element methods in space and a combination of second order central difference and convolution quadrature in time. In Section~\ref{sec:convergenceanalysis} we study the convergence analysis of the mixed time dicretization used which allows us to determine the convergence order of the full scheme. We consider these results for smooth solutions and those with a singularity at $t=0$, since this is expected for equations with fractional time derivatives. Lastly, in Section~\ref{sec:numerics} we show results from the implementation of our scheme to support the theory from the previous section for both smooth and non smooth solutions of \eqref{equ:PDE}.


\section{A damped wave equation model}
\label{sec:dampedwaveequationmodel}
In this section we describe the model with the fractional in time weak damping and prove existence and uniqueness of the solution.  In order to do this we will require some properties of fractional calculus, that we list first.  For more detail on fractional calculus see textbooks by \cite{Diethelm2010} and \cite{OldhamSpanier}.

Before proceeding we introduce the following notation:
\begin{enumerate}[-]
\item  $(\cdot,\cdot)$ denotes the $L^2(\Omega)$ inner product
\item $\|\cdot\| = \|\cdot\|_{L^2(\Omega)}$ denotes the $L^2(\Omega)$ norm
\item  $\| \cdot \|_{p}$ denotes the $H^p (\Omega)$ norm for $p\geq 0$
\item $ \| \cdot \|_{-1} $ denotes the norm of the dual space  $H^{-1}(\Omega) = (H_0^{1}(\Omega))'$ 
\item $\mathcal{C}^k(I;X)$ denotes the space of functions $f: I \rightarrow X$, that are $k$ times continuously differentiable on the time-interval $I$ with an associated Hilbert space $X$. The interval $I$ will either be the closed interval $[0,T]$ or the half-open  interval $(0,T]$.
\item The space of $H^1(\Omega)$ functions with a zero trace on $\partial \Omega$ is denoted by $H^1_0(\Omega)$.
\end{enumerate}

We first give the definition of the Riemann-Liouville fractional integral.

\begin{definition}
\label{def:RLintegral}
For $\beta>0$ the Riemann-Liouville fractional integral of a function $f\in \mathcal{C}[0,T]$ is defined by
\[
I_t^\beta f(t) = \frac{1}{\Gamma (\beta)} \int_0^t (t-\tau)^{\beta -1} f(\tau) d\tau, \qquad t > 0, 
\]
where $\Gamma$ denotes the Gamma function. 
\end{definition} 

The fractional integral $I_t^\beta f(t)$ exists almost everywhere if $f \in L^1(0,T)$, see \cite{Samko,Diethelm2010}, but we will not make use of this reduced smoothness requirement except when $f(t) = t^\mu$ where for $\mu > -1$ we have that
\begin{equation}
\label{equ:integral_poly}
I_t^\beta t^\mu = \frac{\Gamma (\mu+1)}{\Gamma (\mu+\beta +1)} t^{\beta+\mu}.
\end{equation}

The Caputo fractional derivative is obtained by applying the Riemann-Liouville fractional integral to a classical derivative of the function. In the following definition, $\lceil \gamma \rceil$ denotes the ceiling of $\gamma$. 

\begin{definition} \label{def:caputoderiv}
For $\gamma>-1$ and $n=\lceil \gamma \rceil$ the left-sided Caputo fractional derivative of a function $f\in \mathcal{C}^{(n)}[0,T] $ is defined by
\begin{equation*}
\partial_t^\gamma f(t) :=  \begin{cases}
I_{t}^{-\gamma} f(t) & \gamma<0 \vspace{0.5cm}\\
\frac{1}{\Gamma(n-\gamma)}\int_0^t(t-\tau)^{n-\gamma-1} f^{(n)} (\tau) \,d\tau = I_{t}^{n-\gamma} f^{(n)}(t) & \gamma \notin \mathbb{N}, \, \gamma>0 \vspace{0.5cm}\\
\frac{d^n}{dt^n}f(t) & \gamma \in \mathbb{N},
\end{cases}
\end{equation*}
for $t > 0$. 
\end{definition}
Again, an explicit formula can be given in case where $f(t) = t^\mu$ for $\mu > 0$ we have that
\begin{equation}
\label{equ:caputo_poly}
\partial_t^\gamma t^\mu = \left\{
  \begin{array}{ll}
0 &  \gamma  > \mu, \; \mu \in \mathbb{N}_0\\
\frac{\Gamma (\mu+1)}{\Gamma (\mu+1-\gamma)} t^{\mu-\gamma}     & \gamma  \leq \mu. 
  \end{array}
\right.
\end{equation}
\begin{lemma}
\label{lemma:prop}
\begin{enumerate}[(a)]
\item \label{prop:alpha-1}For $\gamma>0$ and any $f \in \mathcal{C}^n[0,T]$ where $n= \lceil \gamma \rceil$
\begin{equation*}
\partial_t^\gamma f(t)= \partial_t^{\gamma-1} \partial_t f(t), \quad t \in [0,T].
\end{equation*}
\item \label{prop:alphaalpha-1} For $\gamma>0$ and any $f \in \mathcal{C}^n[0,T]$ such that $f^{(n-1)}(0)=0$, $n=\lceil \gamma \rceil$, 
\begin{equation*}
\partial_t^\gamma f(t) = \partial_t \partial_t^{\gamma-1} f(t), \quad t \in [0,T].
\end{equation*}
\item \label{prop:semigroup} The semi group property holds for fractional integrals: for any $f \in \mathcal{C}[0,T]$ and $\alpha,\beta \geq 0$
\[
I_t^\alpha I_t^{\beta}f(t) = I_t^{\alpha+\beta}f(t), \quad t \in [0,T].
\]

\end{enumerate}
\end{lemma}

From the definition of the fractional derivative we see that the model \eqref{equ:PDE} is a Volterra integro-differential equation. To proceed with the analysis we will need a result on the corresponding ordinary  differential equation.

\begin{theorem}
\label{thm:ODE}
For $f\in \mathcal{C}[0,T]$, $\gamma \in(-1,1)$, $\lambda \in \mathbb{C}$, $\ag > 0$,  and $u_0, \, v_0 \in \mathbb{R}$ there exists a unique solution to the following initial value problem: Find $u \in \mathcal{C}^2[0,T]$ such that
\begin{align}
 u^{\prime \prime}  + \lambda u + \ag \partial_t^{\gamma+1} u = f &\quad \text{in }[0,T] \nonumber\\
u(0) =& u_0 \label{equ:ODE}\\
u^\prime(0) =& v_0. \nonumber
\end{align}
 Furthermore, for $\gamma \in (-1,0)$
\[
u''(t) = f(t)-\lambda u_0-\frac{\ag}{\Gamma(1-\gamma)}t^{-\gamma}v_0+O(t), 
\]
and for $\gamma \in (0,1)$
\[
u''(t) = f(t)-\lambda u_0-\frac{\ag}{\Gamma(2-\gamma)}(f(0)-\lambda u_0) t^{1-\gamma} + o(t^{1-\gamma}) 
\]
as $t \rightarrow 0^+$.
\end{theorem}
\begin{proof}
For $\gamma = 0$ we omit the proof since in this case the result follows from standard ODE theory.

Using the definition of the Caputo derivative  in \eqref{equ:ODE} gives
\begin{equation}
\label{equ:ODE_with_caputo}
u^{\prime \prime}  + \lambda u + \frac{\ag}{\Gamma (n-\gamma) } \int_0^t (t - \tau)^{n-\gamma -1} u^{(n+1)}(\tau) \, d\tau = f,
\end{equation}
where $n = \lceil \gamma \rceil$.
Next, define $v=u^{\prime \prime}$, i.e., 
\begin{equation}
\label{equ:ODE_u}
u(t) = u_0 + tv_0 + \int_0^t (t-\tau) v(\tau) \, d\tau.
\end{equation}

We  first consider the case $\gamma \in (0,1)$, i.e., $n = 1$. By substituting $v$ into   \eqref{equ:ODE_with_caputo}  we obtain a Volterra integral equation for $v$:
\[
v(t) = f - \lambda u_0 - \lambda t v_0 - \lambda \int_0^t  (t-\tau) v(\tau)\, d\tau - \frac{\ag}{\Gamma (1-\gamma)} \int_0^t (t-\tau)^{-\gamma} v(\tau) \, d\tau.
\]
This Volterra equation for $v$ can be written in the form investigated 
in \cite[Section 6.1.2]{brunner:book}, where the existence of a unique solution $v \in \mathcal{C}[0,T]$ is proved;  see Theorem~\ref{thm:brunner} in the appendix. The uniqueness of $u$ follows by the smoothness assumption $u \in \mathcal{C}^2[0,T]$ and the initial data requirements. Hence $u$ defined by \eqref{equ:ODE_u} is in  $\mathcal{C}^2[0,T]$ and solves the original equation. The asymptotic behaviour of $u''$ follows from Theorem~\ref{thm:brunner}, giving
\[
  \begin{split}    
u''(t) &= f(t)-\lambda u_0-\frac{\ag}{\Gamma(2-\gamma)}v(0) t^{1-\gamma} + o(t^{1-\gamma})\\
&=f(t)-\lambda u_0-\frac{\ag}{\Gamma(2-\gamma)}(f(0)-\lambda u_0) t^{1-\gamma} + o(t^{1-\gamma}).
  \end{split}
\]

If $\gamma \in (-1,0)$, the  equation  is given by
\begin{equation}
\label{equ:ODE01}
u^{\prime \prime}  + \lambda u + \frac{\ag}{\Gamma (-\gamma) } \int_0^t (t - \tau)^{-\gamma-1 } u^{\prime}(\tau) \, d\tau = f.
\end{equation}
Again we wish to let $v = u^{\prime \prime}$ but to do so directly we must write the integral in the above equation as
$$\frac{1}{\Gamma (-\gamma) } \int_0^t (t - \tau)^{-\gamma-1 } u^{\prime}(\tau) \, d\tau = \frac{1}{\Gamma (1-\gamma)} t^{-\gamma} v_0 + \frac{1}{\Gamma (1-\gamma)} \int_0^t (t-\tau)^{-\gamma} v(\tau) \, d\tau.$$
Now we substitute the above equation and \eqref{equ:ODE_u} into the ODE \eqref{equ:ODE01} and we have,
\begin{align*}
v(t) &= f - \lambda u_0 - \lambda t v_0 - \frac{\ag}{\Gamma(1-\gamma)} t^{-\gamma} v_0 - \int_0^t \lambda (t-\tau) v(\tau)\, d\tau \\
&\qquad \qquad- \frac{\ag}{\Gamma (1-\gamma)} \int_0^t (t-\tau)^{-\gamma} v(\tau) \, d\tau. 
\end{align*} 
This is now a Volterra integral equation with a continuous kernel. Hence, see \cite[Theorem~2.1.5]{brunner:book} and  Theorem~\ref{thm:brunner} in the appendix, a unique  solution  $v \in \mathcal{C}[0,T]$ exists and consequently also the solution $u \in \mathcal{C}^2[0,T]$ of the original equation. The asymptotic behaviour of $u''$ again follows from Theorem~\ref{thm:brunner}.
\end{proof}

Next we show that the time-fractional term satisfies a positivity result that will in turn imply its damping properties.

\begin{lemma}\label{lem:positivity_cont}
  Let $f \in \mathcal{C}([0,T]; L^2(\Omega))$, $g \in \mathcal{C}^1([0,T]; L^2(\Omega))$ and let $\gamma \in (0,1)$.  Then the following hold:
  \begin{enumerate}[(a)]
\item  \label{lem:positivity_cont_b}
\[
\int_0^T  (\partial_tI_t^\gamma g(t),g(t))dt \geq \frac{(T/2)^{-1+\gamma}}{\Gamma(\gamma)} \int_0^T \|g(t)\|^2 dt.
\]
  \item \label{lem:positivity_cont_a} 
\[
\int_0^T  (I_t^\gamma f(t),f(t)) dt \geq \frac{(T/2)^{-\gamma}}{\Gamma(1-\gamma)} \int_{0}^T  \|I_t^\gamma f(t)\|^2 dt.
\]
  \end{enumerate}
\end{lemma}
\begin{proof}
The proof of an analogue of (\ref{lem:positivity_cont_b}) for Caputo fractional derivatives is given in \cite{babis_banjai}. The result can also be deduced from Lemma~1.7.2 in \cite{siskova:phd} and Lemma~3.1 in \cite{EndreZener}, which for our case gives
$$
\int_0^T (\partial_t I^\gamma g(t),g(t))\, dt \geq \frac{1}{2\Gamma(\gamma)} \int_0^T \left((T-t)^{\gamma-1}+t^{\gamma-1}\right) \|g(t)\|^2 \,dt .
$$
Minimizing the kernel gives the required result
\begin{equation}
\label{equ:positivity_ourconst}
 \int_0^T  (\partial_tI_t^\gamma g(t),g(t))dt \geq
\frac{(T/2)^{\gamma-1}}{\Gamma(\gamma)}\int_0^T   \|g(\tau)\|^2  d\tau.
\end{equation}

The second part of the lemma follows by using the semigroup property, see Lemma~\ref{lemma:prop}(\ref{prop:semigroup}),
\[
\int_0^T  (I_t^\gamma f(t),f(t)) dt = 
\int_0^T  (I_t^\gamma f(t),\partial_t I_t^{1-\gamma} I_t^\gamma f(t)) dt
\]
and then applying part (\ref{lem:positivity_cont_b}).
\end{proof}

\begin{remark}
 The inequality \eqref{equ:positivity_ourconst} but with a different constant is given in  \cite[Theorem~A.1]{mclean2012}. The constant derived in \cite[Theorem~A.1]{mclean2012}, denoted $C_1$,  and the constant we derive in \eqref{equ:positivity_ourconst}, denoted $C_2$, are given by
\begin{equation}\label{eq:Cs}
C_{1}(\gamma, T) = \pi^{1-\gamma} \frac{(1-\gamma)^{1-\gamma}}{(2-\gamma)^{2-\gamma}}\sin\left( \frac{1}{2}\pi\gamma \right)T^{\gamma-1} \quad \text{and} \quad C_{2}(\gamma,T) = \frac{(T/2)^{\gamma-1}}{\Gamma(\gamma)}.
\end{equation}
A numerical test indicates that $C_2 \geq C_1$ for all $\gamma \in (0,1)$, i.e., it is more optimal; see Fig~\ref{fig:mcleanconst}. In this test we set $T=1$ since both constants have the same dependence on $T$.
\end{remark}

\begin{figure}[H]
    \centering
    \includegraphics[scale=0.5]{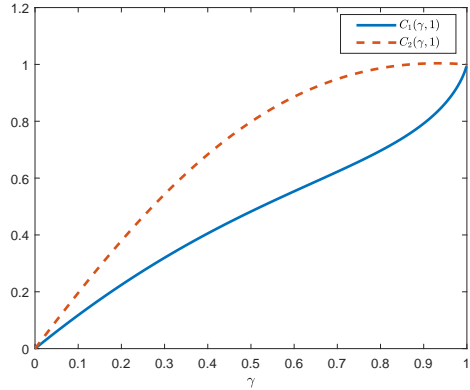}
    \caption{ \label{fig:mcleanconst} Comparison of the constants  \eqref{eq:Cs}. The experiment indicates that $C_2(\gamma,T) \geq C_1(\gamma,T)$, i.e., that $C_2(\gamma,T)$ is more optimal.}
\end{figure}


Next we investigate the existence and uniqueness of the solution of the weak formulation of \eqref{equ:PDE}: Find $u(t) \in H^1_0(\Omega)$ such that
\begin{equation}
  \label{eq:PDE_weak}
  (\partial_t^2u,v)+(\nabla u, \nabla v)+(\ag \partial_t^{\gamma+1}u,v) = (f,v), \qquad
\text{for all } v \in H^1_0(\Omega).
\end{equation}

\begin{theorem}
\label{thm:exist01}
Given $u_0 \in H_0^1(\Omega)$, $v_0 \in L^2 (\Omega) $, $f \in L^2(0,T; L^2(\Omega))$ and $\gamma \in (-1,0)$ a unique weak solution $u \in L^\infty(0,T;H_0^1(\Omega))$ of \eqref{eq:PDE_weak} exists and we have that $\partial_t^{\gamma+1} u \in L^2(0,T;L^2(\Omega))$, $\partial_t^2 u \in L^2(0,T;H^{-1}(\Omega))$ and $\partial_t u \in L^\infty (0,T;L^2(\Omega))$.
\end{theorem}
\begin{proof}
  We follow the standard proof for the wave equation as described, e.g., in \cite{evanspdes,LionsM}, with modifications required due to the fractional damping term. 

Let  $\{\omega_k \}_{k=1}^\infty$ be an orthogonal basis of $H_0^1(\Omega)$ and an orthonormal basis of $L^2(\Omega)$ with corresponding eigenvalues $\{\lambda_k \}_{k=1}^\infty$. We look for $u_m$ of the form
$$u_m (t) = \sum_{k=1}^m d_m^k(t) \omega_k$$
satisfying 
\begin{equation}
  \label{eq:gal_disc}
\left( \partial_t^2 u_m , \omega_k \right)+ \left( \nabla u_m , \nabla \omega_k \right) + \left(\ag \partial_t^{\gamma+1} u_m , \omega_k \right) = \left( f , \omega_k \right), 
\end{equation}
and
\begin{equation}
\label{equ:newics}
 d_m^k (0) = \left( u_0 , \omega_k \right)  \; \text{ and } \;   \frac{d}{dt}d_m^{k}(0) = \left( v_0 , \omega_k \right),
\end{equation}
for $k = 1,2,\dots,m$. As this problem is equivalent to 
\[
\partial_t^2{d_m^k} (t) + \lambda_k d_m^k (t)+\ag \partial_t^{\gamma+1} d_m^k (t) = ( f, \omega_k ),
\]
Theorem~\ref{thm:ODE} shows that a unique solution $d_m^k \in \mathcal{C}^2[0,T]$ exists for all $k \leq m$. 

Testing \eqref{eq:gal_disc} with $\partial_t u_m$, integrating in time, and using Lemma~\ref{lem:positivity_cont}\ref{lem:positivity_cont_a} we obtain 
\[
E(T; u_m) \leq E(0; u_m)+\int_0^T (f,\partial_t u_m) \,dt,
\]
where the energy is given by $E(t;v) = \frac12 \|\partial_t v\|^2+\frac12 \|\nabla v\|^2$. Applying Cauchy Schwarz inequality, the definition of the energy and the Gronwall inequality in the usual way we obtain that the energy is bounded independently of $m$
\begin{equation}
\label{equ:energybound}
E(T;u_m) \leq C \left( \| \nabla u_0 \|^2 + \| v_0 \|^2 + \int_0^T \|f \|^2 \, dt \right) 
\end{equation}
for a constant $C = C(T)$. 

In the standard way, see  \cite{evanspdes}, we obtain a bound on $\partial_t^2 u_m$
\begin{equation}
\label{equ:proof1}
\int_0^T \left\| \partial_t^2 u_m \right\|_{-1}^2\, dt \leq C \left( \| \nabla u_0 \|^2 + \int_0^T  \| f\|^2+\left\| \partial_t^{\gamma+1} u_m \right\|^2 \, dt    \right).
\end{equation}
It remains to bound the term containing the fractional derivative.

Recalling that $\gamma+1 \in (0,1)$, then using the definition of the Caputo derivative, see Definition~\ref{def:caputoderiv}, Young's convolution inequality \cite{beckner} and \eqref{equ:energybound} we deduce that 
\begin{align}\label{eq:youngs}
\int_0^T \left\| \partial_t^\gamma \partial_t u_m \right\|^2\, dt 
&\leq C \int_0^T t^{-\gamma -1} dt \int_0^T  \left\| \partial_t u_m \right\|^2 \,dt \nonumber\\
&\leq C \int_0^T \left\| \partial_t u_m \right\|^2\,dt \nonumber\\
&\leq C \left( \left\| \nabla u_0 \right\|^2 + \left\| v_0 \right\|^2 + \left\| f \right\|_{L^2(0,T; L^2(\Omega))}^2 \right).
\end{align}
 
 Returning to \eqref{equ:proof1} we obtain the bound
$$\int_0^T \left\| \partial_t^2 u_m \right\|_{-1}^2 \leq C \left(  \left\| \nabla u_0 \right\|^2 + \left\| v_0 \right\|^2 + \left\| f \right\|_{L^2(0,T; L^2(\Omega))}^2 \right).$$
Hence $u_m$ has a subsequence that converges weakly to a $u$ in the following spaces
\[
  \begin{split}    
u &\in  L^\infty \left( 0,T ; H_0^1(\Omega) \right),\;
\partial_t u \in L^\infty \left( 0,T ; L^2(\Omega) \right),\;\\
\partial_t^2 u &\in L^2 \left( 0,T ; H^{-1}(\Omega) \right),\;
 \partial_t^{\gamma+1} u \in L^2 \left( 0,T ; L^2(\Omega) \right).
  \end{split}
\]
Further, in the usual way $u$ is a weak solution of the original problem \eqref{equ:PDE}.

Finally we will show that $u$ is a unique weak solution. To do so we let $f=0$, $u_0 = 0$ and $v_0=0$ and show that the weak solution is $u \equiv 0$.
For a fixed $s > 0$, let 
\[
v(t) = \left\{ \begin{array}{ll}
\int_t^s u(\tau) \, d\tau &0\leq t \leq s\\
0 & s < t \leq T
\end{array} \right. .
\]
Testing \eqref{equ:PDE} with $v$ we obtain
\[
0 = \left( \partial_t^2 u ,v \right)+ \left( \nabla u , \nabla v \right)  + \left( \ag\partial_t^{\gamma+1} u , v \right).
\]
Then integrating with respect to time,
$$
0 = \int_0^s \left( \partial_t^2 u ,v \right)+ \left( \nabla u , \nabla v \right)  + \left( \ag\partial_t^{\gamma+1} u , v \right) \, dt = -\int_0^s \left( \partial_t u , \partial_t v \right) - \left( \nabla u , \nabla v \right)  + \left( \ag \partial_t^{\gamma } u ,\partial_t v \right) \, dt
$$
where we used Lemma~\ref{lemma:prop}\ref{prop:alphaalpha-1}. Now by the definition of $v$,
$$
 \int_0^s \frac{1}{2} \partial_t \| u \|^2\, dt  - \int_0^s \frac{1}{2} \partial_t \| \nabla v \|^2\, dt = - \int_0^s \left( \ag\partial_t^{\gamma} u , u \right) \, dt
 $$
and thus
$$
\frac{1}{2} \|u(s)\|^2 + \frac{1}{2} \| \nabla v(0) \|^2 = - \int_0^s \left( \ag \partial_t^{\gamma} u , u \right) \, dt \leq 0
$$
implying $u \equiv 0$ as $\ag>0$.
\end{proof}

We now prove the corresponding result for $\gamma \in (0,1)$.

\begin{theorem}
\label{thm:exist12}
Given $u_0 \in H^2(\Omega) \cap H^1_0(\Omega)$, $v_0 \in H_0^1 (\Omega) $, $\partial_tf \in L^2(0,T; L^2(\Omega))$ and $\gamma \in (0,1)$ a unique weak solution $u \in L^\infty(0,T;H_0^1(\Omega))$ of \eqref{eq:PDE_weak} exists and we have that $\partial_t^{\gamma+1} u \in L^2(0,T;L^2(\Omega))$, $\partial_t^2 u \in L^2(0,T;H^{-1}(\Omega))$ and $\partial_t u \in L^\infty (0,T;L^2(\Omega))$.
\end{theorem}
\begin{proof}
  The only difficulty in extending the proof of Theorem~\ref{thm:exist01} to the case of $\gamma \in (0,1)$ is the bound on $\partial_t^{ \gamma+1}$ in \eqref{eq:youngs}. To circumvent this problem we note that due to the additional smoothness assumption on $f$, the solution $u_m \in \mathcal{C}^2[0,T]$ of \eqref{eq:gal_disc} satisfies the time differentiated equation
\[
 \partial_t^3 u_m -\Delta \partial_tu_m +\ag \partial_t \partial_t^{\gamma+1} u_m = \partial_tf .
\]
Now testing with $\partial_t^2 u_m$ and using Lemma~\ref{lem:positivity_cont}(\ref{lem:positivity_cont_b}) we obtain the energy bound
\[
E(t;\partial_tu_m) \leq C \left(E(0;\partial_tu_m)  + \int_0^t \|\partial_tf \|^2 \, dt \right).
\]
As
\[
  \begin{split}    
E(0;\partial_t u_m) &= \frac12 \|\partial_t^2 u(0)\|^2+ \frac12 \|\nabla \partial_t u(0)\|^2\\
&= \frac12 \|\Delta u_0+f(0)\|^2 + \frac12 \|\nabla v_0\|^2
  \end{split}
\]
we have obtained a bound on $\partial_t^2u$ and consequently using again Young's inequality
\[
\int_0^t \|\partial_t^{\gamma+1}u\|^2dt \leq C\int_0^t \|\partial_t^2u\|^2dt
\leq  C\left(\|\Delta u_0\|^2+\|f(0)\|^2 +\|\nabla v_0\|^2  + \int_0^t \|\partial_tf \|^2 \, dt\right).
\]
Using this bound in the proof of Theorem~\ref{thm:exist01} we obtain the result.

\end{proof}

\begin{remark}
  The fractional Zener wave equation investigated in \cite{EndreZener} is of a similar form to the Szabo equation with $\gamma \in (0,1)$; see equation (2.5) in \cite{EndreZener}. With similar approach to ours, the authors prove uniqueness and existence of the Zener model. The solution in \cite{EndreZener} is understood in a weaker sense, namely lower regularity of the data is required but bounds on the second derivative in time of the solution are not given.
\end{remark}

\begin{remark}\label{rem:smoothness_u}
  As is usual in PDEs with time fractional derivatives, we expect a singularity at $t = 0$ even for smooth data. Namely, due to Theorem~\ref{thm:ODE}  we expect for smooth $f$ and $\gamma \in (0,1)$
\begin{equation}
\label{eq:nonsmooth_general_gampos}
u(x,t)  = u_0(x)+tv_0(x)+t^2w_0(x)+t^{3-\gamma} z_0(x)+o(t^{3-\gamma}).
\end{equation}

The right hand side is then
\[
  \begin{split}
    f(t) &= 
\partial_t^2u -\Delta u+\ag \partial_t^{\gamma+1} u
\\&= 2w_0+ (3-\gamma)(2-\gamma)t^{1-\gamma} z_0-\Delta u_0 + \frac{2\ag}{\Gamma(2-\gamma)} t^{1-\gamma}w_0+ o(t^{1-\gamma}) 
  \end{split}
\]
If $f$ is to be smooth, we need to match the term $t^{1-\gamma}$ by setting  $z_0 = -\frac{2\ag}{\Gamma (4-\gamma)}w_0$.

On the other hand if $\gamma \in (-1,0)$ we have
\[
u(x,t)  = u_0(x)+tv_0(x)+t^2w_0(x)+t^{2-\gamma}z_0(x)+o(t^{2-\gamma}).
\]
The right hand side is then
\[
    f(t) = 2w_0+(2-\gamma)(1-\gamma)t^{-\gamma} z_0-\Delta u_0 + \frac{\ag}{\Gamma (1-\gamma)} t^{-\gamma}v_0+o(t^{-\gamma~}).
\]
If $f$ is to be smooth, we need to match the term $t^{-\gamma}$ by setting  $z_0 = -\frac{\ag}{\Gamma (3-\gamma)}v_0$.


Alternatively, for $u$ to be smooth we would need $f$ to have a singularity of the type $t^{1-\gamma}$ for $\gamma \in (0,1)$ and $t^{-\gamma}$ for $\gamma \in (-1,0)$ with further weaker singularities at $t = 0$.
\end{remark}
\section{Fully discrete system}
\label{sec:fullydiscretesystem}
To obtain the fully discrete system we will use a finite element method in space and a combination of leapfrog and BDF2 based convolution quadrature discretization in time. The motivation for using leapfrog is to obtain an explicit scheme, whereas the main motivation for using convolution quadrature are its excellent stability properties \cite{Lubich1988I,Lubich1986} and  the ability to evaluate it very efficiently \cite{banjai2019,Lubich2006}. An alternative discretization of the fractional time derivative is the L1 scheme \cite{SunWu,LinXu}. This scheme also has the required stability property, i.e., it preserves the positivity property of the fractional derivative; see \cite[Lemma~3.1]{SunWu}. However, to the best of our knowledge correction terms for the L1 scheme are not available except for specific equations such as subdiffusion; see \cite{Yan2018}. 

\subsection{Spatial semidiscretization}
To discretize in space, we make use of a piecewise linear Galerkin finite element method. Namely let $V^h$  be a family of finite dimensional subspaces of $H^1_0(\Omega)$ parametrized by the meshwidth $h > 0$. We assume that these spaces satisfy the following approximation property 
\[
\inf_{v^h \in V^h}\|v-v^h\|_{1} \leq C h \|v\|_{2}
\]
and 
\[
\inf_{v^h \in V^h}\|v-v^h\| \leq C h^2 \|v\|_{2}
\]
for some constant $C> 0$ independent of $h$. Furthermore, we assume that an inverse inequality holds 
\begin{equation}
  \label{eq:inv_in}
  \sup_{v \in V^h} \|\nabla v\| \leq \Cinv h^{-1}\|v\|,
\end{equation}
for some constant $\Cinv$.

The semidiscrete problem then  reads: Find $u^h(t) \in V^h$ such that
\[
  (\partial_t^2u^h,v)+(\nabla u^h, \nabla v)+(\ag\partial_t^{\gamma+1}u^h,v) = (f,v), \qquad
\text{for all } v \in V^h.
\]
We denote projections of  the initial data $u_0$ and $v_0$ onto $V^h$ by $\u_0$ and $\vh_0$ respectively; the projections will be specified later on.
\subsection{Time discretization}

In time we discretize using the explicit, leapfrog scheme 
\[
  \tfrac1{\tstep^2}(\u_{n+1}-2\u_n+\u_{n-1},v)+(\nabla \u_n, \nabla v)+(\ag \w_n,v) = (f(t_n),v),
\]
where $t_n = n\tstep$ for some fixed time step $\tstep >0$, $\u_n \in V^h$ is an approximation of $u(t_n)$ and $w_n$ is an approximation of $\partial_t^{\gamma+1}u(t_n)$ which we describe next. 

First of all we define the central difference operator
\[
 \sder \u(t_n) := \left\{ \begin{array}{ll}
\vh_0 & n=0\vspace{.2cm}\\
\frac{1}{2\tstep}(\u_{n+1}-\u_{n-1}) &  n\geq 1
\end{array} \right. .
\]
We extend the definition also to continuous functions $g \in C[0,T]$ with a given time derivative at $t = 0$:
\[
\sder g(t) := \left\{ \begin{array}{ll}
\partial_t g(0) & t=0\vspace{.2cm}\\
\frac{1}{2\tstep}(g(t+\tstep)-g(t-\tstep)) &  t \geq \tstep\vspace{.2cm}\\
\frac1{\tstep}(\sder g(\tstep)-\partial_tg(0))t+\partial_tg(0) &  t \in [0,\tstep]
\end{array}\right.,
\]
i.e., $\sder g(t)$ is exact at $t = 0$, the central difference quotient for $t \geq \tstep$ and is the linear interpolant of $\sder g$ for $t \in [0,\tstep]$.

It remains to discretize the fractional derivative $\partial_t^\gamma$. The formula should be computable efficiently and should retain the positivity property of the fractional derivative as described in Lemma~\ref{lem:positivity_cont}. All this is satisfied by convolution quadrature introduced by \cite{Lubich1986}, which we describe next.

Convolution quadrature (CQ) is based on an $A(\theta)$-stable linear multistep method. For a continuous function $g$, the CQ formula for $\partial_t^\gamma$ is given by
\begin{equation}
  \label{eq:cq_conv}
  \tder{\gamma} g(t_n) := \sum_{j = 0}^n \omega_{n-j} (g(t_j)-\chi_\gamma g(0)) = \sum_{j = 0}^n \omega_j (g(t_n-t_j)-\chi_\gamma g(0)),  
\end{equation}
where
\[
\chi_\gamma = \left\{
  \begin{array}{cc}
    0& \gamma \in (-1,0]\\
    1 & \gamma \in (0,1)
  \end{array}
\right.
\]
and $\omega_j$ are convolution weights defined below. The term $\chi_\gamma \omega_n g(0)$ above  is added to the standard definition of convolution quadrature in order to correct for the fact that we are using CQ to compute Caputo  rather than  Riemann-Liouville fractional derivatives. 

Next, we define the convolution weights. As central differences are approximations of order 2, we restrict our discussion to second order, BDF2 based CQ.  The corresponding convolution weights are then given by
\[
\left(\frac{\delta(\zeta)}{\tstep}\right)^\gamma =  \sum_{j = 0}^\infty \omega_j \zeta^j, \quad \delta(\zeta) = \frac32-2\zeta+\frac12\zeta^2.
\]

We can also define the approximation at intermediate values of $t$ by
\begin{equation}
  \label{eq:CQ_for_all_t}  
    \tder{\gamma} g(t) := \sum_{j \geq 0; t_j \leq t} \omega_j (g(t-t_j)-\chi_\gamma g(0)) = \sum_{j = 0}^\infty \omega_j (g(t-t_j)-\chi_\gamma g(0)),
\end{equation}
where we define $g(t) \equiv 0$ for $t < 0$.  Further, it is clear that definition \eqref{eq:cq_conv} extends to sequences by
\[
\tder{\gamma} g(t_n) := \sum_{j = 0}^n \omega_{n-j} (g_j-\chi_\gamma g_0) = \sum_{j = 0}^n \omega_j (g_{n-j}-\chi_\gamma g_0).  
\]
From \cite[Theorem~2.1]{Lubich2004} we have the useful estimate
\begin{equation}
  \label{eq:weights_approx}
  |\omega_n-\tstep t_n^{-\mu-1}| \leq Ct_n^{-\mu-1-p}\tstep^{p+1}, \quad p = 0,1,2,
\end{equation}
for $t \in (0,T]$.

As mentioned above, most of the results in the literature analyse  CQ as an approximation to Riemann-Liouville derivatives \cite{Lubich1986}.
 
However, as the Caputo and Riemann-Liouville derivatives of order $\gamma$ are equivalent for $\gamma \in (-1,0)$ and for  $\gamma \in (0,1)$ are equivalent if $g(0) = 0$, we can deduce from \cite[Theorem~2.2]{Lubich2004} the following result.

\begin{lemma}\label{lem:CQ_pol}
  Let $\gamma \in (-1,1)$, $\tstep \in (0,\bar \tstep)$ be the time step for some sufficiently small $\bar \tstep>0$, and $g(t) = t^\beta$ for $\beta \in \mathbb{R}$. Then for BDF2 based CQ it holds
\begin{equation}
\label{eq:CQ_pol}
\left|\tder{\gamma} g(t) - \partial_t^\gamma g (t) \right| \leq C\left\{ 
\begin{array}{ll}
0 & \beta = 0, \gamma \in (0,1)\\
t^{-\gamma -1} \tstep, & \beta = 0, \gamma \in (-1,0)\\
 t^{-\gamma + \beta - p} \tstep^p, &  \beta \geq  1,
\end{array}
\right.
\end{equation}
and $p = 1,2$.
\end{lemma}
\begin{proof}
Note that for $\beta = 0$ and $\gamma \in (0,1)$ we have $\tder{\gamma} g \equiv \partial_t^{\gamma} g \equiv 0$.  The remaining cases follow directly from \cite[Theorem~2.2]{Lubich2004}.
\end{proof}

Next we add correction terms that integrate lower order terms exactly and do not destroy the convergence for the higher order terms.  For $\gamma \in (-1,0)$ it is sufficient to correct for constant functions, whereas for $\gamma \in (0,1)$ we will need to correct for linears as well.

Correction terms  were introduced by Lubich in \cite{Lubich1986} and are of the form
\begin{equation}
\label{equ:correction0_gampos}
\ctder{\gamma} g(t) = \sum_{j \geq 0; t_j \leq t} \omega_{j} g(t-t_j) +   w_{0}(t) g(0)+w_1(t)g(t_1).
\end{equation}
Here, the correction terms $w_{j}$ are chosen so that
\[
 w_0(t) =  \left\{  \begin{array}{ll}
\frac{t^{-\gamma}}{\Gamma (1-\gamma)} - \sum\limits_{j \geq 0; t_j \leq t} \omega_{j}, & \text{for } \gamma \in (-1,0)\vspace{.2cm}\\
 - \sum\limits_{j \geq 0; t_j \leq t} \omega_j - w_1(t), & \text{for } \gamma \in (0,1)
\end{array}
\right.
\]
\[
 w_1(t) = \left\{
   \begin{array}{ll}
0 & \text{for } \gamma \in (-1,0)\vspace{.2cm}\\
\tstep^{-1}\left(\frac{t^{1-\gamma} }{\Gamma (2-\gamma)} - \sum\limits_{j \geq 0; t_j \leq t}  (t-t_j)\omega_j\right),      & \text{for } \gamma \in (0,1).
   \end{array}
\right.
\]
Furthermore, denote $w_{nj} : = w_j(t_n)$ and hence
\[
\ctder{\gamma} g(t_n) = \sum_{j=0}^n \omega_{n-j} g(t_j) +   w_{n0} g(0) +  w_{n1} g(t_1).
\]
Note that by definition $\tder{\gamma}$ already contains the correction for constants if $\gamma \in (0,1)$. In that case the above just adds the correction for linear functions.

\begin{lemma}\label{lem:CQ_pol_corr}
  Under the conditions of Lemma~\ref{lem:CQ_pol}  it holds that for $g(t) = t^\beta$
\begin{equation}
\label{eq:CQ_pol_corr}
\left|\ctder{\gamma} g(t) - \partial_t^\gamma g (t) \right| \leq C\left\{ 
\begin{array}{ll}
0 & \beta = 0,\\
0 & \beta = 1, \gamma \in (0,1)\\
t^{-\gamma+\beta-2}\tstep^2 & \beta \geq 1, \gamma \in (-1,0)\\
t^{-\gamma+\beta-2}\tstep^2+t^{-\gamma-1}\tstep^{\beta+1} & \beta > 1, \gamma \in (0,1).
\end{array}
\right.
\end{equation}
with $C > 0$ independent of $\tstep \in (0,\bar \tstep)$.
\end{lemma}
\begin{proof}
For $\gamma < 0$, the proof follows directly from Lemma~\ref{lem:CQ_pol} and definition of $\ctder{}$. For $\gamma > 0$, Lemma~\ref{lem:CQ_pol} and the definition of $w_1(t)$ imply that $w_{1}(t) \leq  Ct^{-\gamma-1}\tstep$ and hence $w_{n1}g(t_1)  = t^{-\gamma-1}O(\tstep^{1+\beta})$ for $g(t) = t^\beta$ and $\beta > 1$. This in turn implies the final missing result.
\end{proof}

Using the above shorthand for the CQ approximation, we can write the fully discretized systems as
\begin{equation}
  \label{eq:fully_discrete}
    \tfrac1{\tstep^2}(\u_{n+1}-2\u_n+\u_{n-1},v)+(\nabla \u_n, \nabla v)+(\ag\tder{\gamma} \sder \u(t_n),v) = (f(t_n),v)
  \end{equation}
or when including the correction
\begin{equation}
  \label{eq:fully_discrete_corr}
    \tfrac1{\tstep^2}(\u_{n+1}-2\u_n+\u_{n-1},v)+(\nabla \u_n, \nabla v)+(\ag\ctder{\gamma} \sder \u(t_n),v) = (f(t_n),v),
  \end{equation}
$n = 1, \dots, N-1$. The coupling of the two time discretizations is similar to the FEM-BEM coupling in \cite{banjai15}. As both of the above schemes are explicit,  we will see during the course of the analysis that the following CFL condition is required
\begin{equation}
  \label{eq:CFL}
  \tstep \leq \frac{\sqrt{2}h}{\Cinv}.
\end{equation}

It remains to describe the choice of initial data $\u_0, \u_1$. To do this we require  the Ritz projection, denoted by $R_h : H^1_0(\Omega) \rightarrow V^h$ and defined by 
\[
(\nabla R_h u, \nabla v) = (\nabla u, \nabla v) \qquad \text{for all } v \in V^h
\]
and the $L^2$ projection is denoted by $P_h: L^2(\Omega) \rightarrow V^h$
\[
(P_h u,  v) = (u,v) \qquad \text{for all } v \in V^h.
\]
We have the approximation  property,  see, e.g.,  \cite{larssonbook}, 
\begin{equation}
\label{equ:larssonL2}
\|
R_h u - u \| \leq C h^s \| u\|_s \quad \text{ for } s \in [1,2] \text{ and } \forall u \in H^s(\Omega) \cap H_0^1(\Omega).
\end{equation}
We then define the initial data by
\begin{equation}
  \label{eq:initial_data}  
\u_0 = R_hu_0 \quad\text{and}\quad \u_1 = R_h\left(u_0+\tstep v_0\right)+\frac12\tstep^2 P_h\partial_t^2u(0).
\end{equation}
Using the PDE and the fact that $\partial_t^2 u$ is continuous we see that $P_h \partial_t^2 u(0) \in V^h$ is the solution of
\[
(P_h\partial_t^2u(0), v) = (f(0),v)-(\nabla u_0,\nabla v) \quad \text{for all } v \in V^h.
\]
Note also that by definition 
\[
\sder \u(0) = R_h \partial_tu(0) = R_hv_0.
\]
 
The reason for using the mixed approximation $\tder{\gamma}\sder$ instead of a fully CQ approximation $\tder{\gamma+1}$ is to conserve the sign of the damping term.

\begin{lemma}\label{lem:herglotz}
  Given a sequence $v_0, \dots, v_N \in L^2(\Omega)$, we have
  \begin{enumerate}[(a)]
  \item For $\gamma \in (-1,0)$
\[
\sum_{n = 0}^N(\tder{\gamma}v(t_n), v_n) \geq 0
\]
and
\[
\sum_{n = 0}^N(\ctder{\gamma}v(t_n), v_n) \geq -\sum_{n = 0}^N \omega_{n0}(v_0,v_n).
\]
\item For $\gamma \in (0,1)$
\[
\sum_{n = 0}^N(\tder{\gamma}v(t_n), v_n) \geq -\sum_{n = 0}^N \omega_n(v_0,v_n)
\]
and
\[
\sum_{n = 0}^N(\ctder{\gamma}v(t_n), v_n) \geq -\sum_{n = 0}^N \omega_{n0}(v_0,v_n)-\sum_{n = 0}^N \omega_{n1}(v_1,v_n).
\]
  \end{enumerate}
\end{lemma}
\begin{proof}
The proof follows directly from the frequency domain estimate
\[
\Re (s^{\gamma} v,v) \geq 0
\]
for $v \in L^2(\Omega)$ and the Herglotz theorem \cite[Theorem~2.3]{banjai15}.
\end{proof}

\section{Convergence analysis}
\label{sec:convergenceanalysis}
We start with analyzing the error of the mixed approximation $\tder{\gamma}\sder$ when applied to $t^{\beta}$. We will first need a simple technical lemma.

\begin{lemma}\label{lem:sum}
  Let $t_n = n\tstep$ with $\tstep > 0$ and $t_N = T$. Given  $\eta \neq -1$ there exists a constant $C >0$ independent of $T$ and $\tstep$ such that
\[
\tstep \sum_{n = 1}^N t_n^\eta \leq C \max(T^{\eta+1},\tstep^{\eta+1}).
\]
\end{lemma}
\begin{proof}
Consider first the case $\eta \leq 0$. Then $t^\eta$ is a decreasing function  and
\[
    \tstep \sum_{n=1}^N t_n^\eta \leq  \tstep^{\eta+1}+\int_{\tstep}^T t^\eta dt 
    = \frac{1}{\eta+1}\left(T^{\eta+1}+\eta \tstep^{\eta+1}\right).
\]
If $\eta \geq 0$, then $t^\eta$ is an increasing function and we have instead
\[
    \tstep \sum_{n=1}^N t_n^\eta \leq  \int_{\tstep}^{T+\tstep} t^\eta dt
    = \frac1{\eta+1}\left((T+\tstep)^{\eta+1}-\tstep^{\eta+1}\right).
    \]
    Hence in both cases the sum is bounded by $ C \max(T^{\eta+1},\tstep^{\eta+1})$ with the constant depending on $\eta$.
\end{proof}

\begin{lemma}\label{lem:mixed_poly}
  For $g(t) = t^\beta$ and $\beta =0,1$ or $\beta \in [2,3]$ we have
\[
\left|  \tder{\gamma} \sder g(t)  - \partial_t^{\gamma+1} g(t) \right| \leq  C
  \left\{
  \begin{array}{ll}
0 & \beta = 0,\\
0 & \beta = 1, \gamma \in (0,1)\\
    t^{-\gamma-1}\tstep & \beta = 1, \gamma \in (-1,0)\\
t^{-\gamma+\beta-3}\tstep^2 & \beta\in \{2,3\}\\
t^{-\gamma+\beta-3}\tstep^2 +t^{\beta-3}\tstep^{\min(2-\gamma,2)}& \beta \in (2,3),
  \end{array}
\right.
\] 
and constant $C> 0$ independent of $\tstep \in (0,\bar \tstep)$ for some small enough $\bar \tstep > 0$.
\end{lemma}

\begin{proof}
Throughout the proof, $C> 0$ denotes a generic constant allowed to change from one step to another.

First note that the error is 0 for $\beta = 0$ since $\partial_t g \equiv \sder g \equiv 0$. 
For $\beta > 0$, split the error as
\begin{equation}
\label{equ:A}
\left| \tder{\gamma} \sder g(t)  - \partial_t^{\gamma+1} g(t) \right| \leq \underbrace{\left|\tder{\gamma} (\sder g(t) - \partial_t g(t)) \right| }_{A_\beta} + \underbrace{\left| (\tder{\gamma}-\partial_t^\gamma) \partial_t g(t)\right| }_{B_\beta}.
\end{equation}
As $\sder g(t)-\partial_t g(t) = 0$ for $t \geq 0$ if $\beta = 1, 2$, we have that $A_1 = A_2 =  0$ and for $\gamma > 0$ also $A_3 = 0$. 

For $\beta \in (2,3]$ we apply  Newton's generalized binomial theorem to  $\sder g(t) = \frac1{2\tstep}((t+\tstep)^\beta - (t-\tstep)^\beta)$ and see that for $t > \tstep$
\[
  \begin{split}
\sder g(t) -\partial_t g(t)  &=\sum_{\substack{k = 1 \\k\text{ odd}}}^\infty
    \begin{pmatrix}
      \beta \\k
    \end{pmatrix} t^{\beta-k}\tstep^{k-1}-\beta t^{\beta-1}\\
&= \sum_{\substack{k = 3 \\k\text{ odd}}}^\infty
    \begin{pmatrix}
      \beta \\k
    \end{pmatrix} t^{\beta-k}\tstep^{k-1}\\
&= \tstep^2t^{-3}\sum_{\substack{k = 3 \\k\text{ odd}}}^\infty
    \begin{pmatrix}
      \beta \\k
    \end{pmatrix} t^{\beta-k+3}\tstep^{k-3}\\
&\leq \tstep^2t^{-3} \sum_{k = 0}^\infty
    \begin{pmatrix}
      \beta \\k+3
    \end{pmatrix} t^{\beta-k}\tstep^{k}\\
&\leq C\tstep^2t^{-3} \sum_{k = 0}^\infty
    \begin{pmatrix}
      \beta \\k
    \end{pmatrix} t^{\beta-k}\tstep^{k}\\
&= C \tstep^2t^{-3}(t+\tstep)^\beta\leq C \tstep^2 t^{\beta-3}.
  \end{split}
\]
For $t \in [0,\tstep]$ we have instead
\[
  \begin{split}
\sder g(t) -\partial_t g(t) &= 2^{\beta-1}\tstep^{\beta-2}t-\beta t^{\beta-1} = O(\tstep^{\beta-1}).
\end{split}
\]
As the convolution weights satisfy, see \eqref{eq:weights_approx}, $|\omega_j|\leq C \tstep t_j^{-\gamma-1}$, $j \geq 1$, with $|\omega_0| \leq C\tstep^{-\gamma}$
we have  for $t \in [t_n,t_{n+1})$, $n \geq 1$, 
\[
  \begin{split}    
A_\beta &\leq C\tstep^2\left(|\omega_0|t^{\beta-3}+|\omega_n|\tstep^{\beta-3}+\sum_{j = 1}^{n-1}|\omega_{j}|(t-t_j)^{\beta-3}\right)\\
&\leq C\tstep^2\left(\tstep^{-\gamma}t^{\beta-3}+t_n^{-\gamma-1}\tstep^{\beta-2}+\tstep\sum_{j = 1}^{n-1}t_j^{-\gamma-1}(t-t_j)^{\beta-3}\right)\\
&\leq C\tstep^2\left(\tstep^{-\gamma}t^{\beta-3}+t_n^{-\gamma-1}\tstep^{\beta-2}+t_n^{\beta-3}\max(1,\tstep^{-\gamma})\right) = C(t^{\beta-3}\tstep^{\min(2-\gamma,2)}) , \quad t \geq \tstep,
  \end{split}
\]
where we have used \cite[Lemma 5.3]{Lubich1988I} to bound the discrete convolution in the following way:
\[
  \begin{split}
    \tstep \sum_{j = 1}^{n-1} t_j^{-\gamma-1}(t-t_j)^{\beta-3} &\leq
\tstep^{\beta-3-\gamma} \sum_{j = 1}^{n-1} j^{-\gamma-1}(n-j)^{\beta-3}\\
&\leq C\tstep^{\beta-3-\gamma} n^{\max{(-\gamma-1,\beta-3,\beta-\gamma-3)}}\\
&\leq C \tstep^{-\gamma}t_n^{\beta-3} n^{\max{(-\gamma-\beta+2,0,-\gamma)}}\\
&\leq C t_n^{\beta-3}\max(1,\tstep^{-\gamma}),
  \end{split}
\]
where $C$ as always is allowed to depend on $T$. For $t \in (0,\tstep)$ we obtain a similar bound as  $A_\beta \leq C\tstep^{-\gamma+\beta-1} \leq Ct^{\beta-3}\tstep^{2-\gamma}$.

For $\beta = 0$, $B_0 = 0$ and for $\beta = 1$ or $\beta \geq 2$ we have from \eqref{eq:CQ_pol} with  $\partial_t g = \beta t^{\beta-1}$ that
\[
B_\beta \leq C \left\{
  \begin{array}{ll}
    0 & \beta = 1, \gamma \in (0,1)\\
    t^{-\gamma-1}\tstep& \beta = 1, \gamma \in (-1,0)\\
t^{-\gamma+\beta-3}\tstep^2 & \beta \geq 2.
  \end{array}
\right.
\]

Combining all the cases gives the stated result.
\end{proof}


We now prove the corresponding lemma for the corrected quadrature.

\begin{lemma}\label{lem:mixed_poly_corr}
  For $g(t) = t^\beta$ and $\beta =0,1$ or $\beta \in [2,3]$ we have
\[
\left|  \ctder{\gamma} \sder g(t)  - \partial_t^{\gamma+1} g(t) \right| \leq  C
  \left\{
  \begin{array}{ll}
0 & \beta = 0,1\\
0 & \beta = 2, \gamma \in (0,1)\\
t^{-\gamma+\beta-3}\tstep^2 & \beta \in\{2,3\}, \gamma \in (-1,0)\\
t^{-\gamma}\tstep^2+t^{-\gamma-1}\tstep^3 &\beta = 3, \gamma \in (0,1)\\
t^{-\gamma+\beta-3}\tstep^2 & \beta > 2, \beta \neq 3,\gamma \in (-1,0)\\
t^{-\gamma+\beta-3}\tstep^2+t^{-\gamma-1}\tstep^\beta+t^{\beta-3}\tstep^{\min(2,2-\gamma)} & \beta \in (2,3),  \gamma \in (0,1)\\
  \end{array}
\right.
\] 
for $t \in (0,T]$, $C > 0$ independent of time step $\tstep \in (0,\bar \tstep)$ for some small enough $\bar \tstep > 0$.
\end{lemma}
\begin{proof}
  The proof is along the same lines as the proof of Lemma~\ref{lem:mixed_poly}, where using Lemma~\ref{lem:CQ_pol_corr} we note that
\[
B_\beta \leq C\left\{ 
\begin{array}{ll}
0 & \beta = 1,\\
0 & \beta = 2, \gamma \in (0,1)\\
t^{-\gamma+\beta-3}\tstep^2 & \beta \geq 2, \gamma \in (-1,0)\\
t^{-\gamma+\beta-3}\tstep^2+t^{-\gamma-1}\tstep^{\beta} & \beta > 2, \gamma \in (0,1).
\end{array}
\right.
\]
and as before $A_\beta = 0$ for $\beta = 1,2,3$ and $A_\beta = t^{\beta-3}O(\tstep^{\min(2,2-\gamma)})$ for $\beta \in (2,3)$.
\end{proof}

We next investigate the error for functions that are smooth for $t > 0$, but may have a singularity at $t = 0$.

\begin{lemma}
\label{lemma:mixederror}
There exists a constant $C >0$ independent of $\tstep \in (0,\bar \tstep)$ for some small enough $\bar \tstep > 0$, but depending on $\gamma$ and $T$ such that  
\begin{enumerate}[(a)]
\item For $\gamma \in (-1,0)$ and $t \in [\tstep,T]$ if $u \in \mathcal{C}^{3}(0,T]$
\begin{align*}
    \left|  \tder{\gamma} \sder u(t)  - \partial_t^{\gamma+1} u(t) \right| \leq  C &\left\{\partial_t u(\tstep) t^{-\gamma-1}\tstep\right.
\\&\left.+ \left(u^{(2)}(\tstep) t^{-\gamma-1}+\int_\tstep^t(t-\tau)^{-\gamma-1}| u^{(3)}(\tau)|d\tau\right)\tstep^2\right\}.
\end{align*}

\item For $\gamma \in (0,1)$ and $t \in [\tstep,T]$ if  $u \in \mathcal{C}^{4}(0,T]$
\[
  \begin{split}    
    \left|  \tder{\gamma} \sder u(t)  - \partial_t^{\gamma+1} u(t) \right| \leq &  C \left( u^{(2)}(\tstep) t^{-\gamma-1}\right.\\ &+\left.u^{(3)}(\tstep)t^{-\gamma}+\int_\tstep^t (t-\tau)^{-\gamma}| u^{(4)}(\tau)|d\tau\right)\tstep^2.
  \end{split}
\]

\end{enumerate}
\end{lemma}
\begin{proof}
Let $\hat p = \lceil\gamma\rceil+3$ and consider the Taylor expansion at $t = \tstep$ of $u$ with the Peano kernel remainder

\[
     u(t) = \sum_{k = 0}^{\hat p-1}\frac{u^{(k)}(\tstep)}{k!} (t-\tstep)^k +
    R_\tstep(t),
    \]
where 
\[
R_\tstep(t) := \frac{1}{(\hat p -1)!}  \int_\tstep^T (t-\tau)_+^{\hat p-1} u^{(\hat p)}(\tau) d \tau
\]
and $(t-\tau)_+ = \max(t-\tau,0)$ and $t \in [\tstep,T]$.
The first term is a polynomial to which we can directly apply Lemma~\ref{lem:mixed_poly}.  With the integral remainder we proceed as follows.  As $(t-\tau)_+^{\hat p-1} \in \mathcal{C}^{\lceil \gamma\rceil +1}[0,T]$ we can use the Leibniz integral rule and the composition of convolutions to see that the error due to the remainder also reduces to analysing the error for polynomials:
\[
  \begin{split}    
\left|\tder{\gamma} \sder R_\tstep(t)   - \partial_t^\gamma \partial_t R_\tstep(t) \right| 
\leq &  \frac{1}{(\hat{p}-1)!}  \left|\int_\tstep^t  r_{\hat{p}-1,\gamma} (t-\tau) u^{(\hat{p})}(\tau)d\tau \right|\\
\leq & \frac{1}{(\hat{p}-1)!} \int_0^{t-\tstep}| u^{(\hat{p})}(t-\tau)| \left| r_{\hat{p}-1,\gamma} (\tau) \right| d \tau,
\end{split}
\] 
where $r_{k,\gamma}(t) = \tder{\gamma} \sder g(t)- \partial_t^{\gamma+1} g(t)$, with $g(t)  =t^k$. Hence applying the result of Lemma~\ref{lem:mixed_poly} finishes the proof.

\end{proof}

The corresponding result with the corrected quadrature is stated next.

\begin{lemma}
\label{lemma:mixederror_corr}
 There exists a constant $C >0$ independent of $\tstep \in (0,\bar \tstep)$ for some small enough $\bar \tstep > 0$, but depending on $\gamma$ and $T$ such that  
\begin{enumerate}[(a)]
    \item For $\gamma \in (-1,0)$ and $t \in [\tstep,T]$ if  $u \in \mathcal{C}^{3}(0,T]$ 
\[
    \left|  \tder{\gamma} \sder u(t)  - \partial_t^{\gamma+1} u(t) \right| \leq 
C\left(u^{(2)}(\tstep) t^{-\gamma-1}+\int_\tstep^t(t-\tau)^{-\gamma-1}| u^{(3)}(\tau)|d\tau\right)\tstep^2.
\]
\item  For $\gamma \in (0,1)$ and $t \in [\tstep,T]$, if $u \in \mathcal{C}^{4}(0,T]$
\[
    \left|  \ctder{\gamma} \sder u(t)  - \partial_t^{\gamma+1} u(t) \right| \leq   C \left( u^{(3)}(\tstep) t^{-\gamma}+\int_\tstep^t (t-\tau)^{-\gamma}|u^{(4)}(\tau)|d\tau\right)\tstep^2.
\]
\end{enumerate}
\end{lemma}

We next state the smoothness assumptions we will make on the solution. 

\begin{assumption}\label{ass:u}
Let $u \in \mathcal{C}^2([0,T]; H^1(\Omega))$, $u \in \mathcal{C}^4((0,T]; H^1(\Omega))$. Further, let 
\[
  \|\partial_t^{(k)} u(t)\|_1 \leq C_0(1+c_k t^{2+\alpha-k}) \qquad t > 0,
  \]
 for some constants $C_0>0$, $c_k \geq 0$, $k = 0,\dots,4$, and
   \[
\alpha \geq \left\{\begin{array}{cc}
    -\gamma & \text{if } \gamma \in (-1,0) \\
    1-\gamma & \text{if } \gamma \in (0,1).
\end{array}\right.
\]
\end{assumption}

\begin{remark}
\label{remark:err_smoothness}
  Recall, see Remark~\ref{rem:smoothness_u}, that it is not realistic to assume too much smoothness at $t = 0$. Namely, we expect that $u \sim u_0+v_0 t+w_0t^2+O(t^{2+\lceil\gamma\rceil -\gamma})$.
This behaviour justifies the smoothness assumptions we make on $u$. If $u$ happens to be smoother, i.e., with continuous derivatives of order 3 and 4, we can simply set constants $c_3$ and $c_4$ to zero.
\end{remark}

\begin{theorem}\label{thm:main_disc}
Let $u$ be the solution of \eqref{eq:PDE_weak} and $u_n \in V_h$, $n = 0,\dots,N$, the solution of the fully discrete system \eqref{eq:fully_discrete} under the CFL condition \eqref{eq:CFL}. Assuming that $u$ satisfies Assumption~\ref{ass:u} we have that
\[
\begin{split}
\left\| \frac{u_n - u_{n-1}}{\tstep} - u^\prime \left(t_n - \frac{1}{2}\tstep\right) \right\|+ \left\| \frac{u_n + u_{n-1}}{2} - u \left(t_n - \frac{1}{2}\tstep\right) \right\| \leq& \mathcal{E}_h +\left\| (R_h-I)u'(t_n-\frac12\tstep)\right\|+\mathcal{A}_\gamma\\&+
C(c_3+c_4)\tstep^{1+\alpha}+\tilde C \tstep^2
\end{split}
\]
and
\[ 
\begin{split}
\left\| \frac{u_n + u_{n-1}}{2} - u\left(t_n - \frac{1}{2}\tstep \right) \right\|_{1} 
\leq& \mathcal{E}_h+\left\| (R_h-I)u(t_n-\frac12\tstep)\right\|_1+\mathcal{A}_\gamma\\ &+
C(c_3+c_4)\tstep^{1+\alpha}+\tilde C \tstep^2,
\end{split}
\]
where
\[
\begin{split}
\mathcal{E}_h \leq& \tstep\sum_{j = 1}^{N-1}\|(R_h-I)\partial_t^2u(t_j)\|
+ \left\| (R_h -I)v_0 \right\|  \\&+ \frac{\tstep}{2}\left\| (R_h-I) u^{\prime \prime} (0) \right\| 
+ \tstep \sum_{j = 1}^{N-1}\|(R_h-I) \partial_t^\gamma \partial_t u(t_j)\|
\end{split}
\]
and for $\gamma \in (-1,0)$
\[
\mathcal{A}_\gamma \leq C\|\partial_t R_h u(\tstep)\|\tstep
\]
and for $\gamma \in (0,1)$
\[
\mathcal{A}_\gamma \leq  C\|\partial_t^2 R_h u(\tstep)\|\tstep^{2-\gamma}.
\]
The constants $C, \tilde C>0$ are  independent of $\tstep \in (0,\bar \tstep)$, for \ sufficiently small $\bar \tstep$, both are allowed to depend on $T$ and $\gamma \in (-1,1)\setminus 0$. The constant $C$ is independent of $u$, whereas $\tilde C$ can depend on $c_k$, $k = 0,\dots, 4$.
\end{theorem}
\begin{proof}
Let $\e_n = R_h u(t_n) - \u_n$ be the error which satisfies,
\begin{equation}
\label{equ:enprob}
\frac{1}{\tstep^2}\left(  \e_{n+1} -2 \e_n + \e_{n-1} , w \right) + \left( \nabla \e_n , \nabla w \right) + \left( \ag \tder{\gamma} \sder \e (t_n) , w \right) = \left( \delta_n + \ag\varepsilon_n , w \right), 
\end{equation}
for all $w \in V^h$, where
\[
  \delta_n = \frac{1}{\tstep^2}R_h \left( u(t_{n+1}) - 2u(t_n) + u(t_{n-1}) \right)- \partial_t^2 u(t_n)
\]
and
$$ 
\varepsilon_n =  R_h  \tder{\gamma} \sder u (t_n) - \partial_t^{\gamma+1} u(t_n).
$$
Mimicking the continuous case, we test with $w=\bar{\partial}_t \e_n$ 
\begin{equation}
\label{equ:before_sum}
E_{n+1}^e - E_n^e = \tstep \left( \delta_n + \ag\varepsilon_n, \bar{\partial}_t \e_n \right) - \tstep \left( \ag \tder{\gamma} \sder \e (t_n) , \bar{\partial}_t \e_n \right)
\end{equation}
where 
\begin{align}
E_n^e &= \frac{1}{2} \left\| \frac{\e_n - \e_{n-1}}{\tstep} \right\|^2 + \frac{1}{2} \left( \nabla \e_n, \nabla \e_{n-1} \right)\label{equ:disc_energy}\\
&=\frac{1}{2} \left\| \frac{\e_n - \e_{n-1}}{\tstep} \right\|^2 +\frac{1}{4}\left( \nabla e_n^h, \nabla e_{n-1}^h\right) - \frac{1}{8} \left\| \nabla e_n^h - \nabla e_{n-1}^h \right\|^2  + \frac{1}{8}\left( \left\| \nabla e_n^h \right\|^2 + \left\| \nabla e_{n-1}^h  \right\|^2 \right)\nonumber \\
&\geq \frac12\left(1-\frac{1}{4}\Cinv^2 \tstep^2 h^{-2}\right)\left\| \frac{\e_n - \e_{n-1}}{\tstep} \right\|^2+ \frac{1}{2}\left\| \frac{\nabla \e_n + \nabla \e_{n-1}}{2} \right\|^2.
\end{align}
Under the CFL condition \eqref{eq:CFL}, we have that the energy is nonnegative and bounded below as
\begin{equation}\label{eq:enCFL}
E_n^e \geq \frac{1}{4} \left\| \frac{\e_n - \e_{n-1}}{\tstep} \right\|^2 + \frac{1}{2}\left\| \frac{\nabla \e_n + \nabla \e_{n-1}}{2} \right\|^2.    
\end{equation}

As $\sder e^h(0) = 0$, see \eqref{eq:initial_data}, from Lemma~\ref{lem:herglotz} it follows that
\[
\tstep\sum_{n=1}^{N-1} \left( \tder{\gamma} \sder e^h (t_n), \sder e^h_n \right) \geq 0.
\]
Hence, by taking the sum over $n=1,\dots,J$, $J \leq N-1$, of \eqref{equ:before_sum}  we obtain the estimate
\[
E_J^e \leq E_1^e + \tstep \sum_{n=1}^{J} \left( \delta_n + \ag\varepsilon_n, \bar{\partial}_t \e_n \right).
\]
Using the Cauchy-Schwarz inequality we have,
\begin{equation}\label{eq:discEineq}
E_J^e \leq E_1^e + \tstep \sum_{n=1}^{J} \|\delta_n + \ag\varepsilon_n \| \| \bar{\partial}_t \e_n \| \leq E_1^e +  \left( \tstep \sum_{n=1}^{N-1} \| \delta_n + \ag\varepsilon_n \| \right)^2 + \frac{1}{2} \max_{m  \leq N} E_m^e
\end{equation}
due to $\| \bar{\partial}_t \e_n \|^2 \leq E_{n+1}^e + E_n^e \leq 2\max\limits_m E_m^e$.
Hence
\begin{equation}\label{equ:E_m}
\begin{split}  
\frac{1}{2}\max_{m\leq N} E_m^e &\leq E_1^e + \left( \tstep \sum_{n=1}^{N-1} \| \delta_n + \ag\varepsilon_n \| \right)^2\\
&\leq E_1^e + \left( \tstep \sum_{n=1}^{N-1} \| \delta_n\| + \ag \tstep \sum_{n=1}^{N-1}\|\varepsilon_n \| \right)^2.
\end{split}  
\end{equation}
It remains to bound  $\delta_n$, $\varepsilon_n$ and the initial energy. 
 Lemma~\ref{lem:2nderr} and the triangle inequality imply
\begin{equation}
\label{equ:delta_estimate} 
\begin{split}
 \tstep\sum_{n = 1}^{N-1}\| \delta_n \| &\leq \tstep\sum_{n = 1}^{N-1}\|(R_h-I)\partial_t^2u(t_n)\|+C\tstep^2\int_\tstep^t \|R_h u^{(4)}(\tau)\|d\tau+\tstep\int_0^{2\tstep}\|R_h u^{(3)}(\tau)\|d\tau\\
 &\leq \tstep\sum_{n = 1}^{N-1}\|(R_h-I)\partial_t^2u(t_n)\|+C(1+c_4\tstep^{\alpha-1})\tstep^2+C\tstep^2(1+c_3\tstep^{\alpha-1})\\
 &= \tstep\sum_{n = 1}^{N-1}\|(R_h-I)\partial_t^2u(t_n)\|+C(c_3+c_4)\tstep^{1+\alpha}+ O(\tstep^2).
 \end{split}
\end{equation}

To bound $\varepsilon_n$ we split the error  into two parts
\[
  \left\| \varepsilon_n \right\| \leq \underbrace{\left\| R_h \tder{\gamma} \sder u(t_n) - R_h \partial_t^\gamma \partial_t u(t_n) \right\|}_{(A_n)} + \underbrace{\left\| R_h \partial_t^\gamma \partial_t u(t_n) - \partial_t^\gamma \partial_t u(t_n)\right\|}_{(B_n)}.
\]

Hence by Lemma~\ref{lemma:mixederror}, we have for $\gamma \in (-1,0)$
\begin{equation}  
\label{equ:boundingA_gn}
\begin{split}
 A_n &\leq Ct_n^{-\gamma-1}(\|\partial_t R_h u(\tstep)\|\tstep+\|\partial_t^2 R_h u(\tstep)\|\tstep^2)+C\tstep^2 \int_\tstep^{t_n}(t_n-\tau)^{-\gamma-1}\|R_h u^{(3)}(\tau)\|d \tau\\
 &\leq Ct_n^{-\gamma-1}(\|\partial_t R_h u(\tstep)\|\tstep+C_0(1+c_2\tstep^\alpha)\tstep^2)+C\tstep^2 \int_0^{t_n}(t_n-\tau)^{-\gamma-1}(1+c_3\tau^{\alpha-1})d \tau\\
&\leq Ct_n^{-\gamma-1}(\|\partial_t R_h u(\tstep)\|\tstep+\tstep^2).
\end{split}
\end{equation}
Similarly, for $\gamma \in (0,1)$
\begin{equation}
\label{equ:boundingA_gp}
\begin{split}
A_n &\leq C\tstep^2\left(\|\partial_t^2 R_h u(\tstep)\| t_n^{-\gamma-1}+\|\partial_t^3 R_h u(\tstep)\|t_n^{-\gamma}+ \int_\tstep^{t_n}(t_n-\tau)^{-\gamma}\|R_h u^{(4)}(\tau)\|d \tau\right)\\
&\leq C\tstep^2\left(\|\partial_t^2 R_h u(\tstep)\| t_n^{-\gamma-1}+C(1+c_3\tstep^{\alpha-1})t_n^{-\gamma}+ c_4\tstep^{\alpha-1}t_n^{-\gamma}\right),
\end{split}
\end{equation}
where we used the smoothness assumption and the calculation (for $n > 1$)
\[
\begin{split}
    \int_\tstep^{t_n} (t_n-\tau)^{-\gamma}\tau^{\alpha-2} d\tau 
    &= t_n^{-\gamma+\alpha-1} \int_{\tstep/t_n}^1 (1-\tau)^{-\gamma}\tau^{\alpha-2} d\tau\\
    &= t_n^{-\gamma+\alpha-1} \left(\int_{\tstep/t_n}^{1/2} (1-\tau)^{-\gamma}\tau^{\alpha-2}
d\tau + \int_{1/2}^1 (1-\tau)^{-\gamma}\tau^{\alpha-2} d\tau\right)\\
&\leq C t_n^{-\gamma}\tstep^{\alpha-1}.
\end{split}
\]

Thus from Lemma~\ref{lem:sum} we have that if $\gamma \in (-1,0)$
\[
\begin{split}
 \tstep \sum_{n = 1}^{N-1} \|\varepsilon_n\| \leq &C\|\partial_t R_h u(\tstep)\|\tstep+O(\tstep^2)\\
 &+ \tstep \sum_{n = 1}^{N-1}\|(R_h-I) \partial_t^\gamma \partial_t u(t_n)\|
 \end{split}
\]
and for $\gamma \in (0,1)$
\[
\begin{split}
\tstep \sum_{n = 1}^{N-1} \|\varepsilon_n\| \leq& C\|\partial_t^2 R_h u(\tstep)\|\tstep^{2-\gamma}+(c_3+c_4) \tstep^{\alpha+1} +O(\tstep^2)\\
&+ \tstep \sum_{n = 1}^{N-1}\|(R_h-I) \partial_t^\gamma \partial_t u(t_n)\|.
\end{split}
\]

It remains to  estimate the initial error
\[
E_1^e = \frac{1}{2} \left\| \frac{\e_1 - \e_0}{\tstep} \right\|^2 + \frac{1}{2}(\nabla \e_1, \nabla \e_0).
\]
Recalling  the definition of the initial data \eqref{eq:initial_data} and that $\e_n = R_h u(t_n)-\u_n$ we have 
\begin{align*}
\left\| \frac{\e_1 - \e_0}{\tstep} \right\|&= \left\| R_h \left( \frac{u_0}{\tstep}+v_0 + \frac{\tstep }{2} u^{\prime \prime}(0) + \frac{1}{2\tstep}\int_0^\tstep (t-\tau)^2 u^{\prime \prime\prime} (\tau)d\tau \right)\right. \\ 
&\quad \left.- \left( \frac{u_0}{\tstep}+v_0 + \frac{\tstep }{2} P_h u^{\prime \prime} (0) \right) - \frac{1}{\tstep} \left(R_h u_0 - u_0\right) \right\|\\
&=\left\| R_h \left( v_0 + \frac{\tstep }{2} u^{\prime \prime}(0) + \frac{1}{2\tstep}\int_0^\tstep (t-\tau)^2 u^{\prime \prime\prime} (\tau)d\tau \right) - \left( v_0 + \frac{\tstep }{2} P_h u^{\prime \prime} (0) \right)  \right\|\\
&\leq \left\| R_h v_0 - v_0 \right\|  + \frac{\tstep}{2}\left\| R_h u^{\prime \prime} (0) - P_h u^{\prime \prime} (0) \right\| + \frac{\tstep}2\int_0^\tstep \|u^{\prime \prime\prime} (\tau)\|d\tau\\
&\leq \left\| R_h v_0 - v_0 \right\|  + \frac{\tstep}{2}\left\| (R_h-I) u^{\prime \prime} (0) \right\| + Cc_3 \tstep^{1+\alpha} +O(\tstep^2),
\end{align*}
where we used the properties of the $L^2$ projection to see that $\left\| R_h u^{\prime \prime} (0) - P_h u^{\prime \prime} (0) \right\| \leq \left\| (R_h-I) u^{\prime \prime} (0) \right\|$.
The definition of $R_h$ implies that the second term in the initial error vanishes: $(\nabla \e_1, \nabla \e_0) = 0$. This gives the error in the discrete norm \eqref{eq:enCFL}.

 To present the error in a more classical norm we proceed as follows
\[
\begin{split}
     \left\| \frac{u_n - u_{n-1}}{\tstep} -u'(t_n-\frac12\tstep)\right\| &\leq 
    \left\| \frac{\e_n - \e_{n-1}}{\tstep} \right\|+
    \left\| R_h\frac{u(t_n) - u(t_{n-1})}{\tstep} -u'(t_n-\frac12\tstep)\right\|\\
    &\leq \left\| \frac{\e_n - \e_{n-1}}{\tstep} \right\|+
    \left\| (R_h-I)u'(t_n-\frac12\tstep)\right\|+\frac12 \tstep \int_{t_{n-1}}^{t_n} \|u'''(\tau)\|d\tau\\
        &\leq \left\| \frac{\e_n - \e_{n-1}}{\tstep} \right\|+
    \left\| (R_h-I)u'(t_n-\frac12\tstep)\right\|+Cc_3 \tstep(t_n^{\alpha}-t_{n-1}^{\alpha})+ O(\tstep^2).
\end{split}
\]
Also
\[
\begin{split}
    \left\| \frac{u_n + u_{n-1}}{2} -u(t_n-\frac12\tstep)\right\| &\leq
    \left\| \frac{\e_n + \e_{n-1}}{2} \right\|+
    \left\| R_h\frac{u(t_n) + u(t_{n-1})}{2} -u(t_n-\frac12\tstep)\right\|\\
    &\leq C\left\| \frac{\nabla \e_n + \nabla \e_{n-1}}{2} \right\|+
    \left\| (R_h-I)u(t_n-\frac12\tstep)\right\|+O(\tstep^2),
\end{split}
\]
where we used the Poincar\'e-Friedrichs inequality in the last step. Similarly
\[
    \left\| \frac{u_n + u_{n-1}}{2} -u(t_n-\frac12\tstep)\right\|_1 
    \leq C\left\| \frac{\nabla \e_n + \nabla \e_{n-1}}{2} \right\|+
    \left\| (R_h-I)u(t_n-\frac12\tstep)\right\|_1+O(\tstep^2).
\]
Combining this with \eqref{eq:enCFL} and the estimate in the discrete norm gives the required result.
\end{proof}


We next turn to the corrected scheme.

\begin{theorem}\label{thm:main_disc_corr}
Let $u$ be the solution of \eqref{eq:PDE_weak} and $u_n \in V_h$, $n = 0,\dots,N$, the solution of the corrected fully discrete system \eqref{eq:fully_discrete_corr} under the CFL condition \eqref{eq:CFL}. If $u$ satisfies the smoothness conditions from Theorem~\ref{thm:main_disc}
\[
\begin{split}
\left\| \frac{u_n - u_{n-1}}{\tstep} - u^\prime \left(t_n - \frac{1}{2}\tstep\right) \right\|+ \left\| \frac{u_n + u_{n-1}}{2} - u \left(t_n - \frac{1}{2}\tstep\right) \right\| \leq& \mathcal{E}_h +\left\| (R_h-I)u'(t_n-\frac12\tstep)\right\|\\&+
C(c_3+c_4)\tstep^{1+\alpha}+\tilde C \tstep^2
\end{split}
\]
and
\[ 
\begin{split}
\left\| \frac{u_n + u_{n-1}}{2} - u\left(t_n - \frac{1}{2}\tstep \right) \right\|_{1} 
\leq& \mathcal{E}_h+\left\| (R_h-I)u(t_n-\frac12\tstep)\right\|_1\\ &+
C(c_3+c_4)\tstep^{1+\alpha}+\tilde C \tstep^2,
\end{split}
\]
where
\[
\begin{split}
\mathcal{E}_h \leq& \tstep\sum_{j = 1}^{N-1}\|(R_h-I)\partial_t^2u(t_j)\|
+ \left\| (R_h-I) v_0 \right\|  \\&+ \frac{\tstep}{2}\left\| (R_h-I) u^{\prime \prime} (0) \right\| 
+ \tstep \sum_{j = 1}^{N-1}\|(R_h-I) \partial_t^\gamma \partial_t u(t_j)\|.
\end{split}
\]
The constants $C, \tilde C>0$ are  independent of $\tstep \in (0,\bar \tstep)$, for \ sufficiently small $\bar \tstep$, both are allowed to depend on $T$ and $\gamma \in (-1,1)\setminus 0$. The constant $C$ is independent of $u$, whereas $\tilde C$ can depend on $c_k$, $k = 1,\dots, 4$.
\end{theorem}
\begin{proof}
 The proof is a modification of its non-corrected counterpart. The variation lies in the use of Lemma~\ref{lemma:mixederror_corr} and more critically Lemma~\ref{lem:herglotz}.

In particular using the same notation we have that
\[
\tstep\sum_{n=1}^{N-1} \left( \ctder{\gamma} \sder e^h (t_n), \bar{\partial}_t e^h_n \right) \geq -\tstep\sum_{n = 1}^{N-1}\omega_{n1}(\sder e^h(t_1),\sder e^h(t_n)).
\]
For $\gamma \in (-1,0)$, $w_{n1} = 0$ so the proof can proceed in the same way.
 For $\gamma \in (0,1)$ using $|\omega_{n1}| \leq C \tstep t_n^{-\gamma-1}$ we have
\[
\begin{split}
\tstep\sum_{n=1}^{N-1} \left( \ctder{\gamma} \sder e^h (t_n), \bar{\partial}_t e^h_n \right) &\geq -\tstep^2\sum_{n = 1}^{N-1}t_n^{-\gamma-1}\|\sder e^h(t_1)\|\|\sder e^h(t_n)\|\\
&\geq -C \max_n\|\sder e^h(t_n)\|^2 \tstep^2\sum_{n = 1}^{N-1}t_n^{-\gamma-1}\\
&\geq -C\tstep^{1-\gamma} \max_n\|\sder e^h(t_n)\|^2\\
&\geq -2C\tstep^{1-\gamma} \max_n E_n^e.
\end{split}
\]
If $\tstep$ is small enough so that $2C\tstep^{1-\gamma} < \frac14$, this just changes the constant $\frac12$ in \eqref{eq:discEineq} to $\frac34$ and the proof can proceed using Lemma~\ref{lemma:mixederror_corr} instead of Lemma~\ref{lemma:mixederror}, i.e., the error is the same as in Theorem~\ref{thm:main_disc} but with $\mathcal{A}_\gamma = 0$.
\end{proof}

\section{Numerical Results}
\label{sec:numerics}

\subsection{Smooth solution}
First we consider the problem of approximating solutions to \eqref{equ:PDE} in 1D on the interval $\Omega = [0,1]$  with $h=6\tstep$ using the two schemes \eqref{eq:fully_discrete} and \eqref{eq:fully_discrete_corr}. We construct the right hand side $f$ so that the exact solution is given by
\begin{equation}
\label{eq:exact_solution_smooth1d}
u(x,t) = \left( \sin(24t)+\cos(12t) \right)\sin(\pi x).
\end{equation}
We measure the error in the following norm 
\begin{equation}
\label{equ:energy_error}
\text{Error} = \max_n\left\| \frac{u_n - u_{n-1}}{\tstep} - u^\prime \left(t_n - \frac{1}{2}\tstep\right) \right\|+ \left\| \frac{u_n + u_{n-1}}{2} - u \left(t_n - \frac{1}{2}\tstep\right) \right\|
\end{equation}
to compare with the theoretical results in Theorem~\ref{thm:main_disc} and \ref{thm:main_disc_corr}. 

 Solution is smooth in space, hence the error due to the spatial error is $O(h^2) = O(\tstep^2)$. As the solution is smooth in time we can set $c_3 = c_4 = 0$ in Assumption~\ref{ass:u}. Hence Theorem~\ref{thm:main_disc} gives convergence order $O(\tstep)$ for $\gamma < 0$ and order $O(\tstep^{2-\gamma})$ for $\gamma > 0$, whereas Theorem~\ref{thm:main_disc_corr} predicts $O(\tstep^2)$ for all $\gamma$. In Fig.~\ref{fig:1D} we see that the numerical experiments agree with the predicted convergence rates for various  values of $\gamma$, except that  for $\gamma = 0.25$, Fig.~\ref{fig:1D_smooth_025}, we seem to obtain a higher than expected convergence rate $\mathcal{O}(h^2)$ in contrast to the predicted rate $\mathcal{O}(h^{1.75})$. However, by increasing the value of $\alpha_0$ in \eqref{eq:agamma} the predicted convergence rate becomes visible; see Fig.~\ref{fig:highera0}.

\begin{figure}
\centering
\begin{subfigure}{.5\textwidth}
  \centering
  \includegraphics[scale = 0.45]{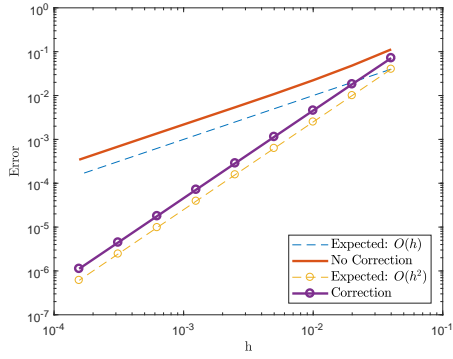}
  \caption{$\gamma = -0.75$}
  \label{fig:1D_smooth_neg075}
\end{subfigure}%
\begin{subfigure}{.5\textwidth}
  \centering
  \includegraphics[scale = 0.45]{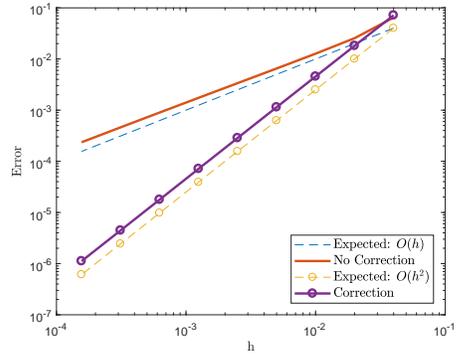}
  \caption{$\gamma = -0.25$}
  \label{fig:1D_smooth_neg025}
\end{subfigure}
\begin{subfigure}{.5\textwidth}
  \centering
  \includegraphics[scale = 0.45]{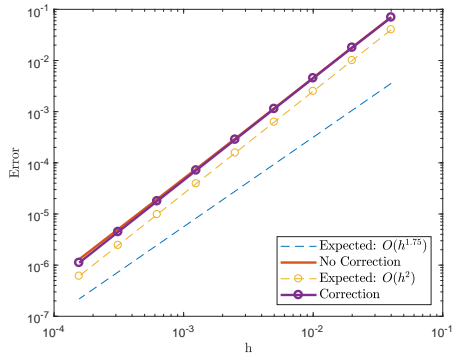}
  \caption{$\gamma = 0.25$}
  \label{fig:1D_smooth_025}
\end{subfigure}%
\begin{subfigure}{.5\textwidth}
  \centering
  \includegraphics[scale = 0.45]{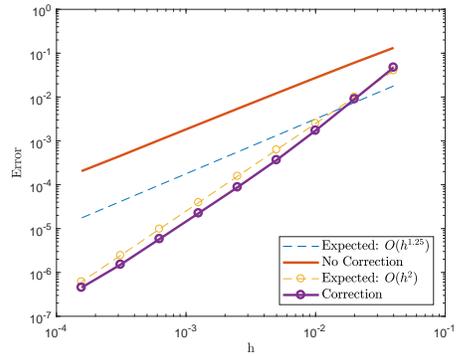}
  \caption{$\gamma = 0.75$}
  \label{fig:1D_smooth_075}
\end{subfigure}
\caption{\label{fig:1D}\small Plot of mesh size $h$ against the maximum error, see \eqref{equ:energy_error}, produced in 1D numerical experiments approximating the exact solution \eqref{eq:exact_solution_smooth1d}, where $h \propto \tstep$. The dashed lines represent the expected convergence rates determined by Theorems \ref{thm:main_disc} and \ref{thm:main_disc_corr}.}
\label{fig:test}
\end{figure}

\begin{figure}
\centering
  \includegraphics[scale = 0.6]{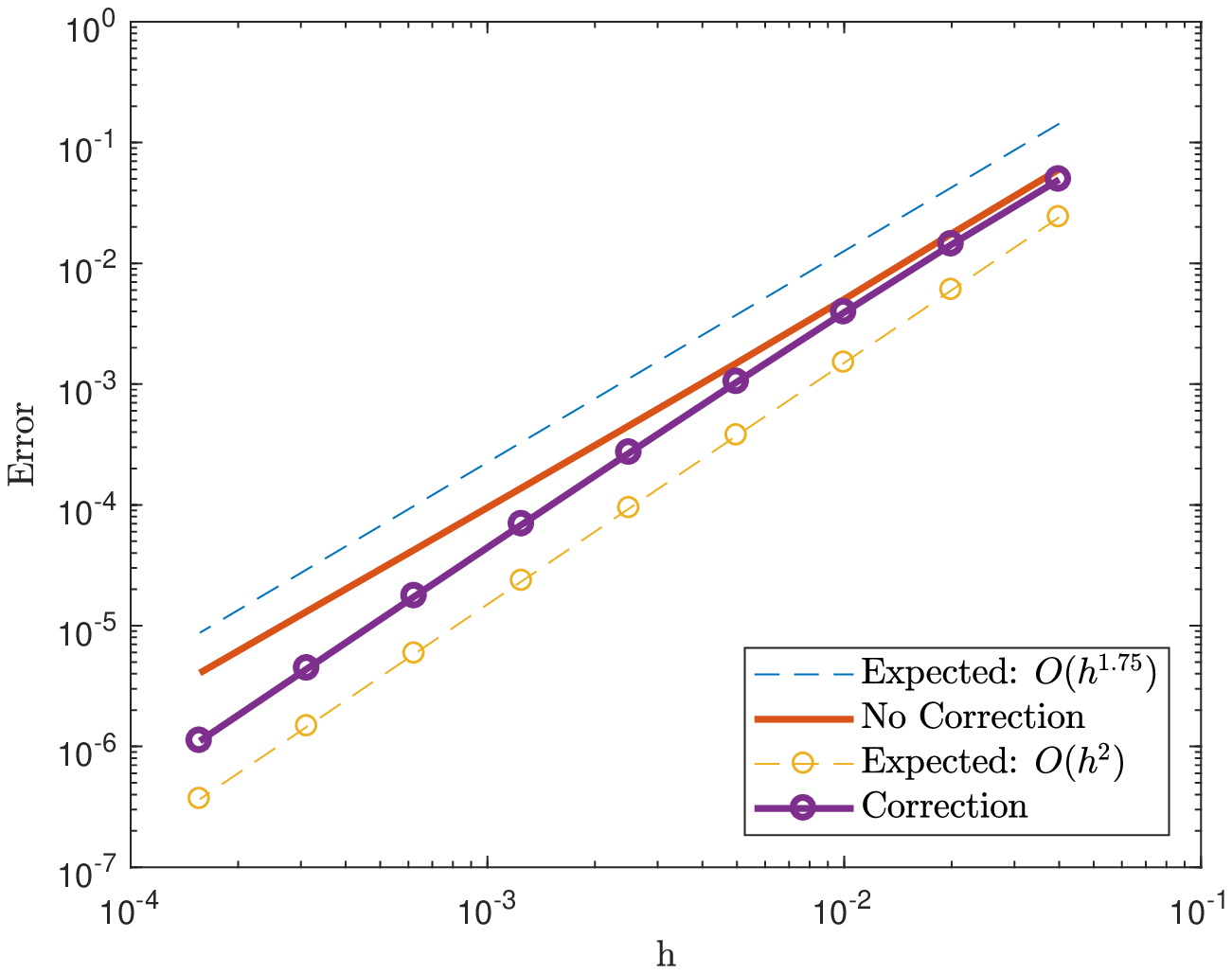}
  \caption{\small A repeat of the experiment in Fig.~\ref{fig:1D_smooth_025} with $\alpha_0$ increased from 1 to 20.}
  \label{fig:highera0}
\end{figure}

Next we perform an experiment in 2D on the domain $\Omega = [-1,1] \times [-1,1]$ with $h=10\tstep$. The right hand side is chosen so that the exact solution is given by
\begin{equation} 
\label{eq:exact_solution_smooth2d}
u(x,t) = \left( \sin(24t)+\cos(12t) \right)\sin(\pi x)\sin(\pi y).
\end{equation}
In 2D experiments we use the $L^2$ error defined as
In Fig.~\ref{fig:2D} we show the convergence of the $L^2$ error
$$
\text{Error} = \max_n\| u_n - u(t_n) \|.
$$
and achieve the expected convergence orders for $\gamma =0.7$ with and without correction terms.

\begin{figure}
\centering
 \includegraphics[scale=0.6]{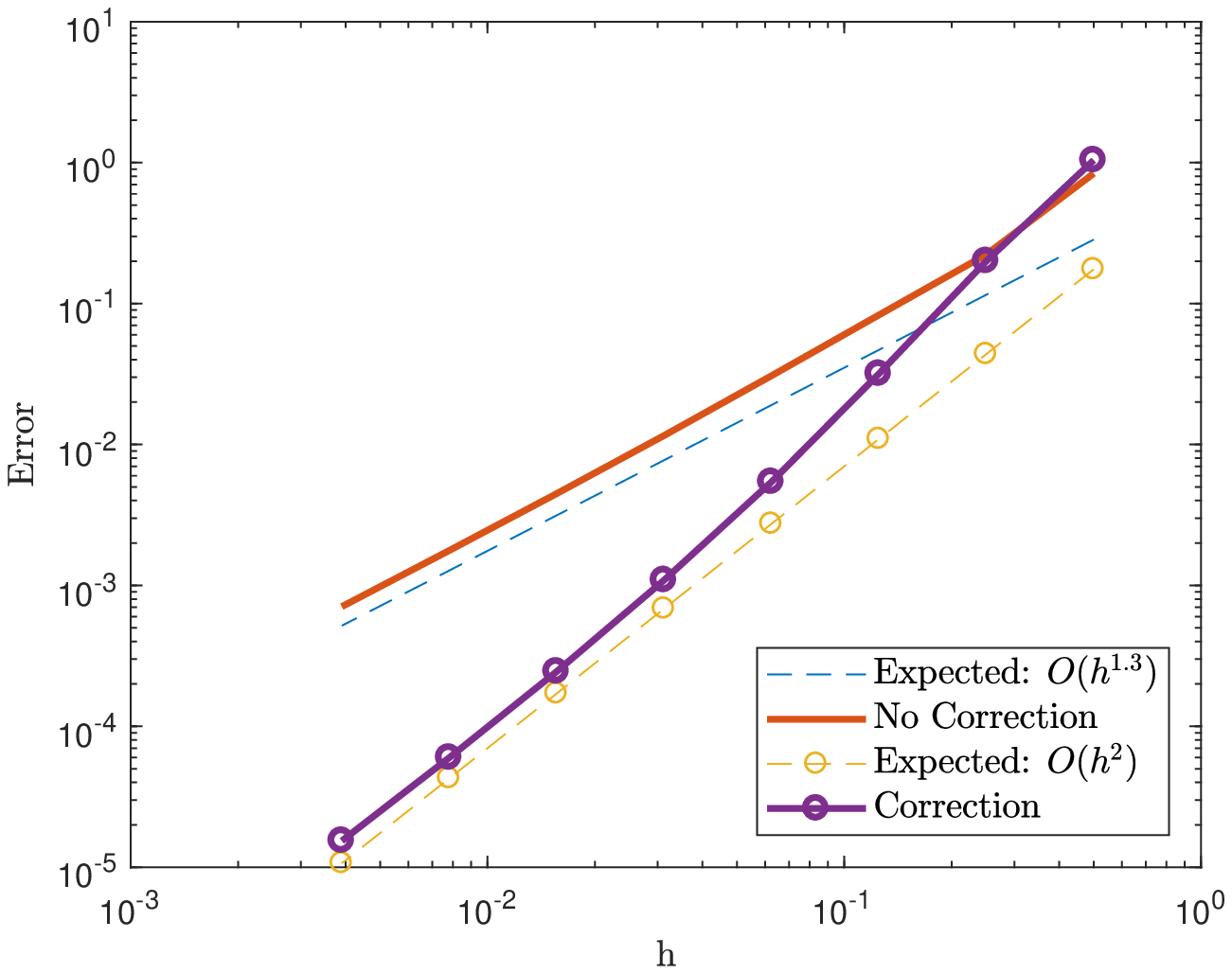}
\caption{\label{fig:2D}\small Plot of mesh size $h$ against the maximum $L^2$~error produced in 2D numerical experiments approximating the exact solution \eqref{eq:exact_solution_smooth2d}, where $h \propto \tstep$. The dashed lines represent the expected convergence rates determined by Theorems \ref{thm:main_disc} and \ref{thm:main_disc_corr}.}
\end{figure}

\subsection{Nonsmooth case}
In the next set of figures we study a more realistic case discussed in Remark~\ref{rem:smoothness_u} in 1D on the domain $\Omega =[0,1]$ with $h=6\tstep$. We now choose the right hand side so that the exact solution is 
\begin{equation}
\label{equ:nonsmooth-exact}
u(x,t) = \left( 1+t + t^2 -\frac{\ag}{\Gamma(3-\gamma+\lceil\gamma\rceil)} t^{2 + \lceil \gamma \rceil-\gamma}\right) \sin(\pi x).
\end{equation}
The error norm is again as in \eqref{equ:energy_error}. Hence the solution satisfies Assumption~\ref{ass:u} with non-zero constants $c_3$, $c_4$ and $\alpha = \lceil \gamma \rceil-\gamma$. Again the solution is smooth in space. 

 The results shown in Fig.~\ref{fig:1D_nonsmooth} generally agree with our claims from Theorem~\ref{thm:main_disc} and Theorem~\ref{thm:main_disc_corr}, except that  in some cases we achieve a higher convergence rate than expected. More specifically, for $\gamma = 0.25$ we observe second order convergence with and without correction terms, and when $\gamma = 0.75$ we have a rate of $\mathcal{O}(\tstep^{1.35})$ with corrected CQ. By increasing $\alpha_0$ we would see the expected convergence rates in these two cases, similarly to the adjustment we see in Fig.~\ref{fig:highera0}.
\begin{figure}
\centering
\begin{subfigure}{.5\textwidth}
  \centering
  \includegraphics[scale = 0.45]{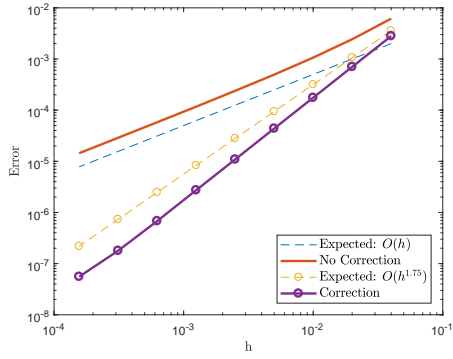}
  \caption{$\gamma = -0.75$}
  \label{fig:1D_nonsmooth_neg075}
\end{subfigure}%
\begin{subfigure}{.5\textwidth}
  \centering
  \includegraphics[scale = 0.45]{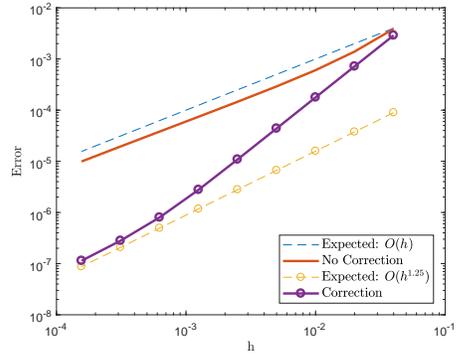}
  \caption{$\gamma = -0.25$}
  \label{fig:1D_nonsmooth_neg025}
\end{subfigure}
\begin{subfigure}{.5\textwidth}
  \centering
  \includegraphics[scale = 0.45]{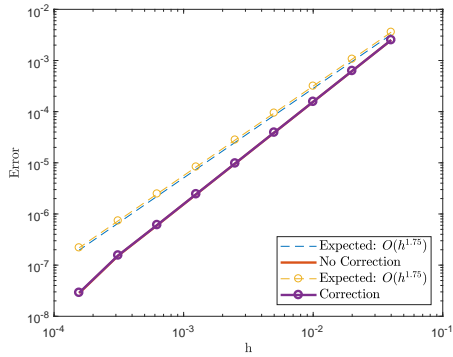}
  \caption{$\gamma = 0.25$}
  \label{fig:1D_nonsmooth_025}
\end{subfigure}%
\begin{subfigure}{.5\textwidth}
  \centering
  \includegraphics[scale = 0.45]{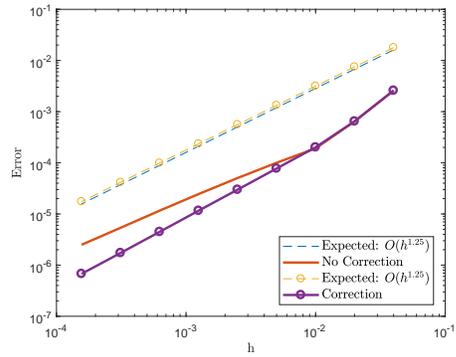}
  \caption{$\gamma = 0.75$}
  \label{fig:1D_nonsmooth_075}
\end{subfigure}
\caption{ \small Plot of mesh size $h$ against the maximum error, see \eqref{equ:energy_error}, produced in 1D numerical experiments approximating the exact solution \eqref{eq:exact_solution_smooth2d}, where $h \propto \tstep$. The dashed lines represent the expected convergence rates determined by Theorems \ref{thm:main_disc} and \ref{thm:main_disc_corr}.}
\label{fig:1D_nonsmooth}
\end{figure}

\subsection{Damping in 2D}
We end the section on numerical experiments, by illustrating the damping effect for the fractional term. Fig.~\ref{fig:damping} shows the profile our approximation of the  solution of the PDE \eqref{equ:PDE} with 
$$
u_0 = e^{-10(x^2+y^2)}, \quad v_0 = 0 \quad \text{and} \quad f=0
$$
at the point $(0,0)$ on the domain $\Omega = [-1,1]\times[-1,1]$. The first plot has no fractional derivative included, i.e., $a_\gamma = 0$,  and the remaining have varying $\gamma$s. In this experiment, the damping effect seems to be strongest for  $\gamma=0.25$. 

\begin{figure}
\centering
\includegraphics[scale=0.5]{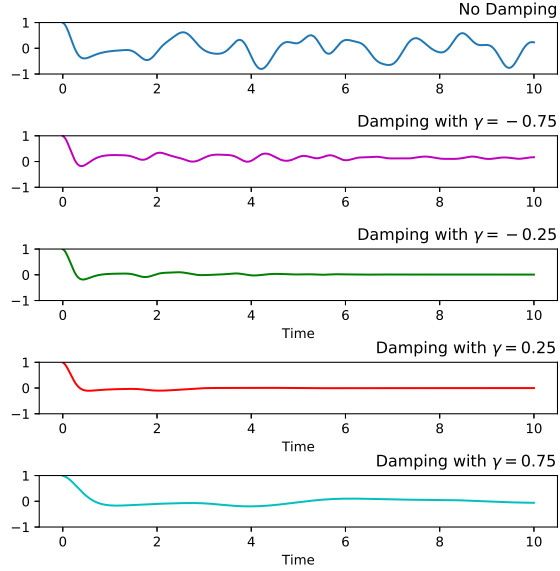}
\caption{\label{fig:damping} \small The profile of a solution to \eqref{equ:PDE} at one point on a 2D mesh with varying levels of damping introduced through changing the order $\gamma$ of the fractional derivative. For the case with no damping we remove the fractional derivative.}
\end{figure}

\bibliographystyle{abbrv}

\appendix

\section{Appendix}

Consider the linear Volterra equation of the second kind
\begin{equation}
  \label{eq:volterra}
  v(t) = g(t)+ \int_0^t(t-\tau)^{-\gamma} k(t-\tau) v(\tau) \, d\tau,
\end{equation}
for $\gamma \in (-1,1)$ and $k \in C[0,T]$. Before stating an existence result, we need a technical lemma the proof of which is elementary.

\begin{lemma}
  Let $g, k \in C[0,T]$. Then for $\beta > -1$
\[
\int_0^t (t-\tau)^\beta k(t-\tau)g(\tau)d\tau = g(0) k(0) \frac{t^{\beta+1}}{\beta+1} + o(t^{\beta+1})  \quad \text{as }t\rightarrow 0^+.
\]
\end{lemma}
\begin{proof}
  As $k$ and $g$ are continuous, there exist constants $C_k>0$, $C_g>0$ such that $|k(t)| \leq C_k$ and $|g(t)| \leq C_g$. Further for every $\varepsilon > 0$ there exists a $\delta > 0$ such that
\[
\max\{|k(t)-k(0)|,|g(t)-g(0)|\} \leq \varepsilon \qquad \text{for } t \in (0,\delta).
\]
Hence
\[
\left|\int_0^t (t-\tau)^\beta k(t-\tau)g(\tau)d\tau - g(0) k(0) \frac{t^{\beta+1}}{\beta+1}\right|\leq (C_k+C_g)\frac1{\beta+1}t^{\beta+1} \varepsilon,
\]
for any $t \in (0,\delta)$.
\end{proof}
\begin{theorem}\label{thm:brunner}
If  $g \in C[0,T]$, then \eqref{eq:volterra} has a unique solution $v \in C[0,T]$. Furthermore,
\[
v(t) = g(t)+g(0)k(0)  \frac{t^{1-\gamma}}{1-\gamma}+o(t^{1-\gamma}) \quad \text{as }t\rightarrow 0^+.
\]
\end{theorem}
\begin{proof}
For $\gamma \in (-1,0)$, the kernel $(t-\tau)^{-\gamma}k(t-\tau)$ is continuous and hence the existence of the continuous solution follows from \cite[Theorem~2.1.5]{brunner:book}.  Whereas, for $\gamma \in (0,1)$, \cite[Theorem~6.1.2]{brunner:book} gives the existence of the continuous solution.

The technical lemma together with  the fact that $v(0) = g(0)$ and \eqref{eq:volterra} gives the form of solution for $t \rightarrow 0^+$.
\end{proof}

%
%
The following technical lemma is needed to investigate the error for the nonsmooth solution.
%

\begin{lemma}\label{lem:2nderr}
Let $g \in \mathcal{C}^4(0,T]$ with $g^{(3)}\in L^1[0,T]$. Then for $t \in [2\tstep,T]$
\[
\left|g''(t)-\frac1{\tstep^2}(g(t+\tstep)-2g(t)+g(t-\tstep))\right|
\leq C \tstep\int_{t-\tstep}^{t+\tstep}|g^{(4)}(\tau)|d\tau
\] 
and for $t = \tstep$ 
\[
\left|g''(t)-\frac1{\tstep^2}(g(t+\tstep)-2g(t)+g(t-\tstep))\right|
\leq C \int_{0}^{2\tstep}|g^{(3)}(\tau)|d\tau,
\]
where $C> 0$ is a constant depending on $C_k$ and $\alpha$.
\end{lemma}
\begin{proof}
Let $t \geq 2\tstep$. Then using the integral form of the remainder
\[
\left|g''(t)-\frac1{\tstep^2}(g(t+\tstep)-2g(t)+g(t-\tstep))\right| \leq C\tstep \int_{t-\tstep}^{t+\tstep} |g^{(4)}(\tau)|d\tau.
\]
This gives the proof for $t \geq 2\tstep$.

Next, for $t > \tstep$ and $s \in [t-\tstep,t+\tstep]$ we have the Taylor expansion
\[
g(s) = g(t)+g'(t)(s-t)+\frac12g''(t) (s-t)^2+\int_t^s(s-\tau)^2g^{(3)}(\tau)d\tau.
\]
Hence
\[
|\frac1{\tstep^2}(g(t+\tstep)-2g(t)+g(t-\tstep)) - g''(t)| \leq C \int_0^{2\tstep}|g^{(3)}(\tau)|d\tau.
\]
Letting $t \to \tstep$ gives the result for $t = \tstep$.
\end{proof}

\section*{Acknowledgements}
Katherine Baker was supported by The Maxwell Institute Graduate School in Analysis and its Applications, a Centre for Doctoral Training funded by the UK Engineering and Physical Sciences Research Council (grant EP/L016508/01), the Scottish Funding Council, Heriot-Watt University and the University of Edinburgh.

We acknowledge discussions with David Sinden formerly of the National Physical Laboratory, UK. Also we gratefully acknowledge comments by Endre S\"uli and the anonymous referees. 

\end{document}